\documentclass[11pt]{amsart}
\usepackage{amsmath,amsfonts,amssymb,mathrsfs}
\usepackage{color}
\usepackage{graphicx}
\title[Convergence of a semi-Lagrangian scheme]{Convergence of a semi-Lagrangian scheme for the BGK model of the Boltzmann equation}
\author{Giovanni Russo}
\address{Dipartimento di Matematica e Informatica, Universita di Catania, Viale Andrea Doria 6, 95125,
Catania, Italia}
\email{russo@dmi.unict.it}
\author{Pietro santagati}
\address{Dipartimento di Matematica e Informatica, Universita di Catania, Viale Andrea Doria 6, 95125,
Catania, Italia}
\email{psantagati@dmi.unict.it}
\author{Seok-Bae Yun}
\address{Department of Mathematical Sciences, KAIST (Korea Advanced Institute of Science and Technology),
373-1 Guseong-dong, Yuseong-gu, Daejeon, 305-701, Korea}
\email{sbyun@kaist.ac.kr}

\begin{document}
\allowdisplaybreaks
\newtheorem{theorem}{Theorem}[section]
\newtheorem{lemma}{Lemma}[section]
\newtheorem{corollary}{Corollary}[section]
\newtheorem{proposition}{Proposition}[section]
\newtheorem{remark}{Remark}[section]
\newtheorem{definition}{Definition}[section]

\renewcommand{\theequation}{\thesection.\arabic{equation}}
\renewcommand{\thetheorem}{\thesection.\arabic{theorem}}
\renewcommand{\thelemma}{\thesection.\arabic{lemma}}
\newcommand{\bbr}{\mathbb R}
\newcommand{\bbt}{\mathbb T}
\newcommand{\bbs}{\mathbb S}
\newcommand{\bbz}{\mathbb Z}
\newcommand{\bn}{\bf n}


\def\charf {\mbox{{\text 1}\kern-.24em {\text l}}}
\begin{abstract}
Recently, a new class of semi-Lagrangian methods for the BGK model of the Boltzmann equation has been introduced \cite{F-R,R-P,Sant}.
These methods work in a satisfactory way either in rarefied or fluid regime. Moreover, because of the semi-Lagrangian feature, the stability property is not restricted by the CFL condition. These aspects make them very attractive for practical applications.
In this paper, we investigate the convergence properties of the method and prove that the
discrete solution of the scheme converges in a weighted $L^1$ norm to the unique smooth solution by deriving an
explicit error estimate.
\end{abstract}
\maketitle

\tableofcontents
%
%
%
%
\section{{\bf Introduction}}
In the kinetic theory of gases, the dynamics of a non-ionized monatomic rarefied gas system is
described by the celebrated Boltzmann equation.
But numerical approximation of the Boltzmann dynamics is a formidable challenge
due mainly to the complicated structure of the collision operator.
Many good numerical techniques have been developed to this end, but often they lead to time consuming computations.

To circumvent these difficulties, Bhatnagar, Gross and Krook \cite{B-G-K}, and independently Welander \cite{W}, proposed
a simplified model for the Boltzmann equation where the collision operator was replaced by a relaxation operator:
\begin{align}
\begin{aligned}\label{main.1}
\displaystyle\frac{\partial f}{\partial t}  + v \cdot \nabla_x f&= \frac{1}{\kappa}(\mathcal{M}(f)-f),
\quad  (x,v,t) \in \bbt^d \times \bbr^d \times \bbr_+, \cr
\displaystyle f(x,v,0) &= f_0(x,v).
\end{aligned}
\end{align}
Here $\bbt^d$ denotes the $d$-dimensional torus and $\displaystyle\frac{1}{k}$ is the collision frequency.
Although the collision frequency takes various forms depending on hypotheses imposed in the derivation
of the model \cite{A-S-Y,Pi-Pu,Stru,Yun1,Yun2,Z-Stru}, we assume in this paper that it is a fixed constant for simplicity.
${\mathcal M}$ denotes the local Maxwellian constructed from the velocity moments of the distribution function $f$:
\[
\displaystyle\mathcal{M}(f)(x,v,t)=\frac{{\rho(x,t)}}{\sqrt{(2\pi T(x,t)})^d}\exp\Big(-\frac{|v- U(x,t)|^2}{2 T(x,t)}\Big),
\]
where
\begin{eqnarray*}
\rho(x,t)&=& \int_{\bbr^d} f(x,v,t)dv,\cr
\rho(x,t)U(x,t)&=&\int_{\bbr^d}f(x,v,t)v dv,\cr
d\rho(x,t)T(x,t)&=& \int_{\bbr^d} f(x,v,t)| v-U(x,t)|^2dv.
\end{eqnarray*}
\indent The BGK model (\ref{main.1}) is computationally less expensive than the Boltzmann equation since it is sufficient to update the macroscopic fields in each time step.
On the other hand, it provides qualitatively correct solutions for the macroscopic moments in fluid regime.
These two aspects, namely, the relatively low computational cost and the correct description of hydrodynamic limit, explain the interest in the BGK model and its variations over the last decades.
It also shares important features with the original Boltzmann equation, such as
the conservation laws and the dissipation of entropy:
\begin{eqnarray}\label{ConservationLaw}
\int \mathcal{M}(f)
\left(\begin{array}{c}
1\\
v\\
|v|^2\\
\end{array}
\right)
=\int f
\left(\begin{array}{c}
1\\
v\\
|v|^2\\
\end{array}
\right) dv
\end{eqnarray}
and
\begin{eqnarray}\label{EntropyDissipation}
\int ({\mathcal M}(f)-f)\log f dv\leq 0.
\end{eqnarray}
The local conservation laws (\ref{ConservationLaw}) leads to a system of hydrodynamic type equations.
\begin{align}
\begin{aligned}
&\frac{d}{dt}\int f dv+\nabla_x\cdot\int v f dv=0,\cr
&\frac{d}{dt}\int f vdv+\nabla_x\cdot\int v\otimes v f dv=0,\cr
&\frac{d}{dt}\int f |v|^2dv+\nabla_x\cdot\int v|v|^2 f dv=0.
\end{aligned}
\end{align}
There are extensive literatures on various topics of the BGK model.
For mathematical analysis, we refer to \cite{Bel,Iss,Mischler,P,P-P,Yun1,Yun2}. For numerical computations,
since there are too many of them, we do not attempt to present a complete set of references.
See \cite{C,Pi-Pu,P-R,F-R,Sant} and references therein.

Recently, a semi-Lagrangian scheme was proposed and tested successfully for various flow problems arising
in gas dynamics \cite{F-R,R-P,Sant}.
The first order version of the method can be written down as follows:
\begin{equation}
f^{n+1}_{i,j,R}=\frac{\kappa}{\kappa+\triangle t}~\widetilde{f}^{n}_{i,j,R}+\frac{\triangle t}{\kappa+\triangle t}~
{\mathcal M}^n_{i,j}(\widetilde{f}^{n}_R),\label{MainScheme_Introduction}
\end{equation}
where $\mathcal{M}^n_{i,j}$ denotes the local Maxwellian defined by
\[
{\mathcal M}^n_{i,j}(\widetilde{f}^{n}_R)=\frac{ {\widetilde\rho}^n_{iR}}{\sqrt{(2\pi \widetilde {T}^n_{i,R}})^N}\exp\Big(-\frac{|v_j-\widetilde {U}^n_{i,R}|^2}{2\widetilde {T}^n_{i,R}}\Big).
\]
The precise definitions of each terms will be given in later sections.
The main feature of the scheme is that even though the relaxation operator is treated
implicitly, computations can be performed explicitly by exploiting the approximate conservation laws in a
very clever way. (See section 2). Therefore, (\ref{MainScheme_Introduction}) enjoys the stability
property of implicit schemes and the low computational cost of explicit schemes at the same time.
Moreover, the semi-Lagrangian treatment of the transport part enables one to
perform the computation over a wide range of CFL numbers. 
In this paper, we study the convergence issue of this scheme and derive an explicit
estimate of the convergence rate measured in a weighted $L^1$ space. As far as we know, this seems to be
the first result on the strong convergence of a fully discretized scheme for nonlinear collisional kinetic equations.

This paper, after introduction is organized as follows. In section 2, we describe the numerical method considered in this paper.
In section 3, we recall relevant existence results. In section 4, we present our main result.
Then several essential estimates to be used in later sections are presented in section 5. In section 6, we derive a
consistent form and obtain error estimates of the remainder terms.
Finally, in section 7, we combine these elements to prove the main theorem.

%
%
%
%
\section{{\bf Description of the numerical scheme}}
For simplicity, we consider one dimensional problem in space and velocity.
We assume constant time step $\triangle t$ with final time $T_f$ and uniform grid in space and velocity with mesh spacing $\triangle x$, $\triangle v$
respectively. We denote the grid points as follows:
\begin{align}
\begin{aligned}\label{gridnodes}
&t^n=n\triangle t,\hspace{1cm}n=1,...,N_t, \cr
&x_i=i\triangle x,\hspace{1.1cm}i=1,...,N_x,\quad (mod ~1),\cr
&v_j=j\triangle v,\hspace{0.95cm}j=-N_v,..,0,..,N_v,
\end{aligned}
\end{align}
where $N_t\triangle t = T_f$, $N_x\triangle x=1$ and $N_v\triangle v=R$.
We also denote the approximate solution of $f(x_i,v_j,t)$ by $f^n_{i,j,R}$.
To describe the numerical scheme more succinctly, we introduce the following
convenient notation.
First, we define $x(i,j)=x_i-\triangle t v_j$. We also set $s=s(i,j)$ to be the index of the spatial node such that
$x(i,j) \in [x_s, x_{s+1})$.
\begin{definition}
We define the reconstructed distribution function $\widetilde f^{n}_{i,j,R}$ as
\begin{equation}
\widetilde f^{n}_{i,j,R}=\frac{x(i,j)-x_s}{\triangle x}~f^{n}_{s+1,j,R}+\frac{x_{s+1}-x(i,j)}{\triangle x}~f^{n}_{s,j,R}.\label{shift}
\end{equation}
\end{definition}
Note that $\widetilde f^n_{i,j,R}$ is the linear reconstruction of $f(x-v\triangle t,v,n\triangle t)$.
\begin{figure}[htbp]
\centering
\includegraphics[angle=-90, scale=0.50]{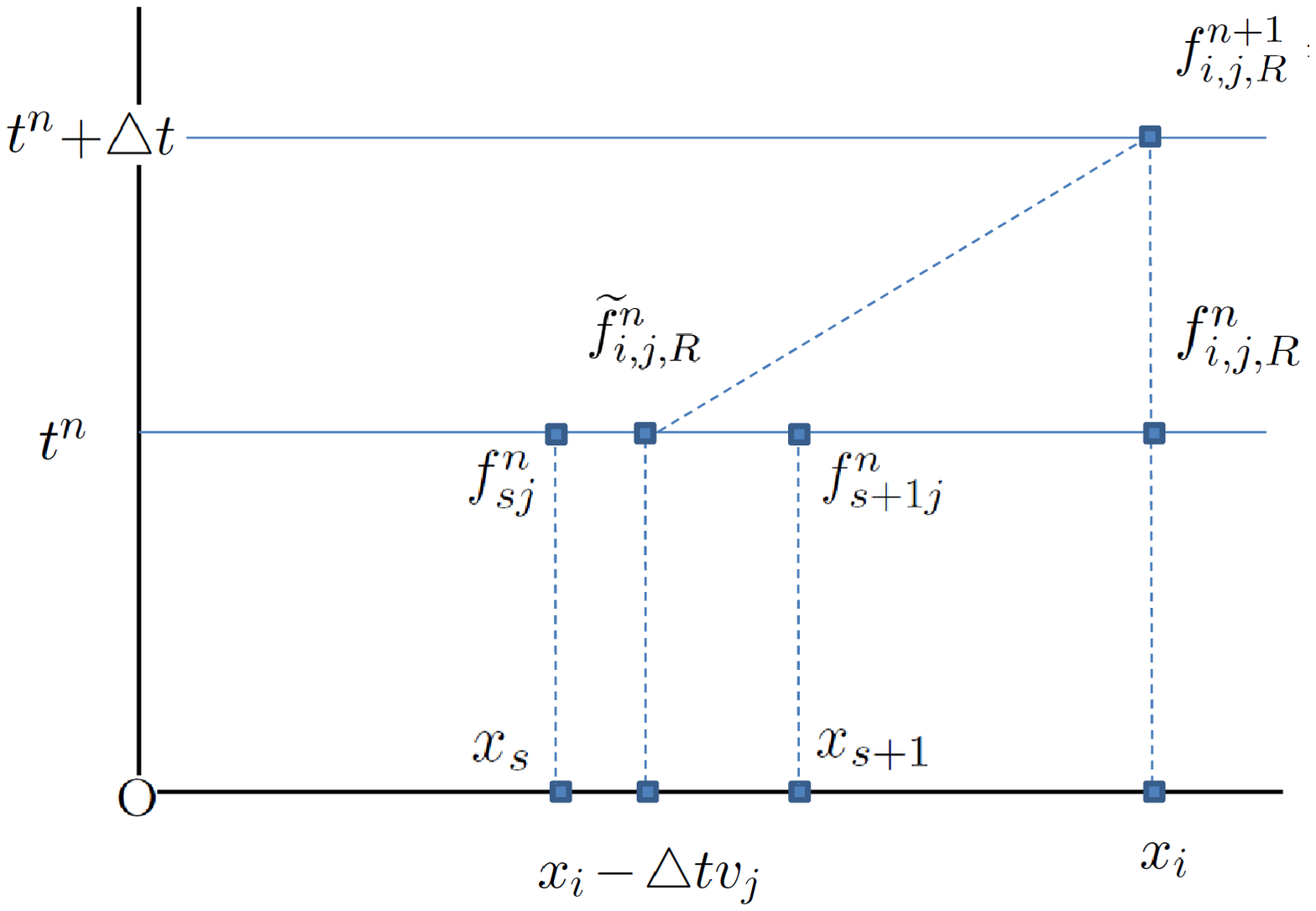}
\caption{Characteristics diagram for positive velocity grid node.}
\end{figure}
The numerical scheme we consider in this paper can now be stated as follows.
%
%
\begin{align}\label{MainScheme}
\begin{aligned}
f^{n+1}_{i,j,R}&=\frac{\kappa}{\kappa+\triangle t}~\widetilde{f}^{n}_{i,j,R}+\frac{\triangle t}{\kappa+\triangle t}~
{\mathcal M}^n_{i,j}(\widetilde{f}^{n}_R),\cr
f^0_{i,j,R}&=\big(f_0\mathcal{X}_{\{|v|\leq R\}}\big)(x_i,v_j),
\end{aligned}
\end{align}
where
\[
{\mathcal M}^n_{i,j}(\widetilde{f}^{n}_R)=\frac{ {\widetilde\rho}^n_{iR}}{\sqrt{(2\pi \widetilde {T}^n_{i,R}})^N}\exp\Big(-\frac{|v_j-\widetilde {U}^n_{i,R}|^2}{2\widetilde {T}^n_{i,R}}\Big),
\]
and
\begin{align}
\begin{aligned}
\widetilde\rho^n_{i,R} &= \sum_{j} \widetilde f^{n}_{i,j,R}\triangle v,\cr
\widetilde\rho^n_{i,R} \widetilde U^n_{i,R} &= \sum_{j} \widetilde f^{n}_{i,j,R} v_j\triangle v,\cr
d\widetilde \rho^n_{i,R} \widetilde T^{n}_{i,R} &= \sum_{j} \widetilde f^{n}_{i,j,R}| v_j-\widetilde {U}_{i,R}|^2\triangle v.
\end{aligned}\label{Conservatives_Original}
\end{align}
We now explain briefly how (\ref{MainScheme}) is derived.
%
%
%
%
The numerical scheme for (\ref{main.1}) is based on the following characteristic formulation of the problem:
\begin{align}
\begin{aligned}\label{characteristic.formulation1}
&\frac{df}{dt}=\frac{1}{\kappa}(\mathcal{M}(f)-f),\cr
&\frac{dx}{dt}=v.
\end{aligned}
\end{align}
Then the time evolution of $f_{i,j}(t)=f(x_i,v_j,t)$ along the characteristic line in the time step $[t_n,t_{n+1}]$ is
presented as
\begin{align}
\begin{aligned}\label{characteristic.formulation2}
&\frac{df_{i,j}}{dt}=\frac{1}{\kappa}(\mathcal{M}(f_{i,j})-f_{i,j}),\cr
&\frac{dx_i}{dt}=v_j.
\end{aligned}
\end{align}
We then discretize (\ref{characteristic.formulation2}) implicitly to obtain
\begin{eqnarray}\label{implicit.euler}
\frac{f^{n+1}_{i,j}-\widetilde{f}^n_{i,j}}{\triangle t}=\frac{1}{\kappa}(\mathcal{M}^{n+1}_{i,j}(f^{n+1}_{ij})-f^{n+1}_{i,j}).
\end{eqnarray}
At this step, (\ref{implicit.euler}) seems to be time consuming, and this is where the clever trick kicks in.
We first observe that conservation of mass, momentum and energy gives for $\phi(v_j)=1,v_j, \frac{1}{2}|v_j|^2$
\begin{eqnarray*}
\frac{\sum_{j}f^{n+1}_{i,j}\phi(v_j)-\sum_{j}\widetilde{f}^n_{i,j}\phi(v_j)}{\triangle t}
&=&\frac{1}{\kappa}\sum_{j}\big(\mathcal{M}^{n+1}_{i,j}(f^{n+1}_{ij})-f^{n+1}_{i,j}\big)\phi(v_j)\cr
&\approx&0,
\end{eqnarray*}
which implies
\[
\sum_{j}f^{n+1}_{i,j}\phi(v_j)\approx\sum_{j}\widetilde{f}^n_{i,j}\phi(v_j).
\]
This in turn gives
\begin{eqnarray}\label{M=tildeM}
\mathcal{M}^{n+1}_{i,j}({f^{n+1}_{i,j}})\approx\mathcal{M}^n_{i,j}({\widetilde{f}^n_{i,j}}).
\end{eqnarray}
We then substitute the above approximation (\ref{M=tildeM}) into (\ref{implicit.euler}) to obtain
\begin{eqnarray}
\frac{f^{n+1}_{i,j}-\widetilde{f}^n_{i,j}}{\triangle t}=\frac{1}{\kappa}(\mathcal{M}^{n}_{i,j}(\widetilde{f}^{n}_{ij})-f^{n+1}_{i,j}).
\end{eqnarray}
Note that the implicit scheme (\ref{implicit.euler}) now can be calculated explicitly.
We then collect relevant terms together to obtain (\ref{MainScheme}).
For more details about (\ref{MainScheme}), we refer to \cite{R-P,Sant}.
%
%
%
%
\subsection{{\bf Extension of the scheme}}
In this section, we extend the discrete distribution function $\{f^{n}_{i,j}\}_{i,j}$ to the whole numerical domain $\bbt_x\times \bbr_v$
and reformulate (\ref{MainScheme}) in accordance with the extension. This allows us to treat the discrete numerical solution and the exact solution
in the same framework.
First we introduce
\begin{eqnarray*}
C_1(x)&=&\sum_{i}x_i~\mathcal{X}_{(x_{i-1}\leq x<x_{i})},\\
C_2(v)&=&\sum_{i}v_i~\mathcal{X}_{(v_{i-\frac{1}{2}}\leq v <v_{i+\frac{1}{2}})}.
\end{eqnarray*}
Here $\mathcal{X}_A$ denotes the usual characteristic function and we used the following convenient notation:
\[
v_{i\pm\frac{1}{2}}=v_i\pm\frac{\triangle v}{2}=\big(i\pm\frac{1}{2}\big)\triangle v.
\]
We now define
\begin{eqnarray*}
{\mathcal A}_{x_i}(x,v)&=&\mathcal{X}_{A_{x_i}}(x,v),\cr
{\mathcal A}^R_{i,j}(x,v)&=&\mathcal{X}_{A_{x_i}}(x,v)\cdot \mathcal{X}_{A^R_{v_j}}(x,v),
\end{eqnarray*}
where
\begin{eqnarray*}
A_{x_i}&=&\{(x,v)|~ C_1(x)=x_i\},\\
A^R_{v_j}&=&\{(x,v)|~ C_2(v)=v_j~\}\cap\{(x,v)|~ |v|\leq R \}.
\end{eqnarray*}
For given sequences $\{a_i\}$ and $\{b_{ij}\}$ defined on grid nodes, we define the following extension operators:
\begin{eqnarray*}
&&E_x(a_{i})(x,v)\equiv\sum_{i,j}\Big(\frac{x-x_{i}}{\triangle x}a_{i+1}+\frac{x_{i+1}-x}{\triangle x}a_{i}\Big){\mathcal A}_{x_i}(x,v),\cr
&&E(b_{ij})(x,v)\equiv\sum_{i,j}\Big(\frac{x-x_{i}}{\triangle x}b_{i+1,j}+\frac{x_{i+1}-x}{\triangle x}b_{i,j}\Big){\mathcal A}^R_{i,j}(x,v).
\end{eqnarray*}
Now the approximate distribution function can be extended to the whole numerical domain as follows.
\begin{eqnarray}\label{ExtendedDistributionFtns}
f^{n}_R(x,v)&\equiv&E(f^n_{i,j,R})(x,v)\nonumber\\
&=&\sum_{i,j}\Big(\frac{x-x_{i}}{\triangle x}f^{n}_{i+1,j}
+\frac{x_{i+1}-x}{\triangle x}f^n_{i,j}\Big){{\mathcal A}^R_{i,j}}(x,v).\label{fnR_Extension}
\end{eqnarray}
Note that $f^n_R(x,v)$ is piecewise constant in the velocity domain and
piecewise linear in the spatial domain.
Using this, we define the macroscopic fields and the local Maxwellian as follows:
\begin{align}
\begin{aligned}
\rho^n_R(x,t)&= \int f^n_R(x,v,t)dv,\cr
\rho^n_R(x,t) U^n_R(x,t) &= \int f^n_R(x,v,t)C_2(v) dv,\cr
d\rho^n_R(x,t) T^n_R(x,t) &= \int f^n_R(x,v,t)|C_2(v)-U^n_R(x,t)|^2 dv
\end{aligned}\label{Conservatives_Reform}
\end{align}
and
\begin{equation}
\mathcal{M}^n(f^n_R)(x,v)\equiv \frac{\rho^n_R(x)}{\sqrt{(2\pi T^n_R(x)}}\exp\Big(-\frac{|C_2(v)-U^n_R(x)|^2}{2\pi T^n_R(x)}\Big).\label{reform_M}
\end{equation}
We also define the reconstructed distribution function $\widetilde{f}^n_R$ as
\begin{eqnarray*}
\widetilde{f}^n_R(x,v)&=&E_{x,v}(\widetilde f^n_{i,j,R})\cr
&=&\sum_{i,j}\Big(\frac{x-x_{i}}{\triangle x}\widetilde f^{n}_{i+1,j}
+\frac{x_{i+1}-x}{\triangle x}\widetilde f^n_{i,j}\Big){{\mathcal A}^R_{i,j}}(x,v),\label{reform_tilde2_PartIII_Chap1}
\end{eqnarray*}
Then the corresponding macroscopic fields and the local Maxwellian are defined analogously:
\begin{align}
\begin{aligned}\label{Conservatives_Reform_PartIII_Chap1}
\widetilde\rho^n_R(x,t)&= \int\widetilde f^n_R(x,v,t)dv,\cr
\widetilde\rho^n_R(x,t)\widetilde U^n_R(x,t) &= \int\widetilde f^n_R(x,v,t)C_2(v) dv,\cr
d\widetilde\rho^n_R(x,t)\widetilde T^n_R(x,t) &= \int\widetilde f^n_R(x,v,t)|C_2(v)-\widetilde U^n_R(x,t)|^2 dv
\end{aligned}
\end{align}
and
\begin{equation}
\mathcal{M}^n(\widetilde f^n_R)(x,v)\equiv \frac{\widetilde \rho^n_R(x)}{\sqrt{(2\pi \widetilde T^n_R(x)}}
\exp\Big(-\frac{|C_2(v)-\widetilde U^n_R(x)|^2}{2\pi \widetilde T^n_R(x)}\Big).\label{reform_M}
\end{equation}
\subsection{{\bf Consistency}}
The following series of lemmas show that the preceding definitions are reasonable.
%
%
%
%
\begin{lemma}\label{PeriodicityofFnR} $f^{n}_R$, $\widetilde f^{n}_R$ are periodic functions with period 1:
\begin{align}
\begin{aligned}
f^{n}_R(x+1,v)&=f^{n}_R(x,v),\\
\widetilde f^{n}_R(x+1,v)&=\widetilde f^{n}_R(x,v).
\end{aligned}
\end{align}
\end{lemma}
\begin{proof}
This follows directly from the periodicity of the spatial domain.
\end{proof}
%
%
%
%
\begin{lemma}\label{ConsistLemma1}For each $x_i$ and $v_j$, the following consistency properties hold.
\begin{align}
\begin{aligned}
f^{n}_R(x_i,v_j)&= f^n_{i,j,R},\\
\widetilde f^n_R(x_i,v_j)&= \widetilde f^n_{i,j,R}.
\end{aligned}\label{reform_tilde_PartIII_Chap1}
\end{align}
\end{lemma}
\begin{proof}
We recall the definition \eqref{fnR_Extension} of $f^{n}_R(x_i,v_j)$ to see
\begin{eqnarray*}
f^{n}_R(x_i,v_j)&=&\sum_{m,\ell}\Big(\frac{x_i-x_{m}}{\triangle x}f^{n}_{m+1,\ell}+\frac{x_{m+1}-x_i}{\triangle x}f^{n}_{m,\ell}\Big)
{\mathcal A}^R_{m,\ell}(x_i,v_j)\\
&=&\frac{x_i-x_{i}}{\triangle x}f^{n}_{i+1,j}+\frac{x_{i+1}-x_i}{\triangle x}f^{n}_{i,j}\\
&=&f^{n}_{i,j}
\end{eqnarray*}
The second statement follows in a similar way.
\end{proof}
%
%
%
\begin{lemma}\label{ConsistLemma2} For macroscopic fields, we have
\begin{align}
\begin{aligned}
\rho^n_R(x_i,t)&= \rho^n_{i,R},\\
\rho^n_R(x_i,t) U^n_R(x_i,t) &= \rho^n_{i,R} U^n_{i,R},\\
\rho^n_R(x_i,t) T^n_R(x_i,t) &= \rho^n_{i,R} T^n_{i,R}
\end{aligned}
\end{align}
and
\begin{align}
\begin{aligned}
\widetilde\rho^n_R(x_i,t)&= \widetilde\rho^n_{i,R},\cr
\widetilde\rho^n_R(x_i,t)\widetilde U^n_R(x_i,t) &= \widetilde\rho^n_{i,R}\widetilde U^n_{i,R},\cr
\widetilde\rho^n_R(x_i,t)\widetilde T^n_R(x_i,t) &= \widetilde\rho^n_{i,R}\widetilde T^n_{i,R}.
\end{aligned}
\end{align}
\end{lemma}
\begin{proof}
We prove the third identity. We observe
\begin{eqnarray*}
&&d\rho^{n}_R(x_i)T^n_R(x_i,t)\cr
&&\hspace{1cm}\equiv\int f^n_R(x_i,v)|C_2(|v|)-U^n_R(x_i,t)|^2dv\cr
&&\hspace{1cm}=\int\sum_{m,\ell}\Big(\frac{x_i-x_{m}}{\triangle x}f^{n}_{m+1,\ell}
+\frac{x_{m+1}-x_i}{\triangle x}f^n_{m,\ell}\Big)|C_2(|v|)-U^n_R(x_i,t)|^2{{\mathcal A}^R_{m,\ell}}(x_i,v)dv\cr
&&\hspace{1cm}=\int\sum_{\ell}\Big(\frac{x_i-x_{i}}{\triangle x}f^{n}_{i+1,\ell}
+\frac{x_{i+1}-x_i}{\triangle x}f^n_{i,\ell}\Big)|C_2(|v|)-U^n_R(x_i,t)|^2{{\mathcal A}^R_{m,\ell}}(x_i,v)dv\cr
&&\hspace{1cm}=\int\sum_{\ell}f^n_{i,\ell}|C_2(|v|)-U^n_R(x_i,t)|^2~{{\mathcal A}^R_{i,\ell}}(x_i,v)dv\cr
&&\hspace{1cm}=\sum_{\ell}f^n_{i,\ell}|v_j-U^n_i|^2\triangle v\cr
&&\hspace{1cm}=d\rho^n_{i,j,R}T^n_{i,j,R}.
\end{eqnarray*}
Other identities can be proved in a similar manner.
\end{proof}
%
%
%
%
\begin{lemma}\label{ConsistLemma3M} The following consistency properties hold for local Maxwellians.
\begin{align}
\begin{aligned}
\mathcal{M}^n(f^n_R)(x_i,v_j)&=\mathcal{M}^n_{i,j}(f^n_{i,j,R}),\\
\mathcal{M}^n(\widetilde f^n_R)(x_i,v_j)&=\mathcal{M}^n_{i,j}(\widetilde f^n_{i,j,R}).
\end{aligned}
\end{align}
\end{lemma}
\begin{proof}
This follows directly from the definition \eqref{reform_M} and Lemma \ref{ConsistLemma2}.
\end{proof}
%
%
%
%
Our main scheme (\ref{MainScheme}) can now be restated as follows
\begin{eqnarray}\label{a.f.reform}
f^{n+1}_{i,j,R}=\frac{\kappa}{\kappa+\triangle t}\widetilde{f}^n_R(x_i, v_i)+\frac{\triangle t}{\kappa+\triangle t}\mathcal{M}^n(\widetilde{f^n_R})(x_i,v_i).
\end{eqnarray}
Applying the extension operator to \eqref{a.f.reform} once more, we obtain the following reformulation of (\ref{MainScheme}):
\begin{theorem}The discrete scheme (\ref{MainScheme}) can be recast in the following form:
\begin{align}
\begin{aligned}\label{refom_scheme}
f^{n+1}_R(x,v)&=\frac{\kappa}{\kappa+\triangle t}\widetilde{f}^n_R(x,v)
+\frac{\triangle t}{\kappa+\triangle t}E(\mathcal{M}^n(\widetilde{f^n_R}))(x,v),\cr
f^0_R(x,v)&=E\big(f_0{\mathcal X}_{\{|v|\leq R\}}\big)(x,v).
\end{aligned}
\end{align}
where we used a slightly abbreviated notation for brevity:
\begin{eqnarray*}
E(\mathcal{M}^n(\widetilde{f^n_R}))(x,v)\equiv E(\mathcal{M}^n(\widetilde{f^n_R})(x_i,v_j))(x,v).
\end{eqnarray*}
\end{theorem}
\subsection{{\bf Notation}}
Before we proceed to the next section, we set some notational conventions.
\begin{itemize}
\item $C$ denotes generic constants.
\item $C_{x,y,..}$ denotes generic constants that depend on $x$, $y$,... but not exclusively.
\item We use the following convention for the $L^1$ norm with polynomial weight and the $L^{\infty}$ norm.
\begin{eqnarray*}
\|f(t)\|_{L^1_q}&=&\int_{\bbr^d\times\bbr^d}f(x,v,t)(1+|v|)^qdxdv,\\
\|f(t)\|_{L^{\infty}}&=&\sup_{x,v}|f(x,v,t)|,\\
\|f(t)\|_{L^{\infty}}&=&\sup_{x}|f(x,t)|.
\end{eqnarray*}
\item We introduce the following notation for weighted $L^{\infty}$-Sobolev norms for smooth or approximate solutions.
\begin{eqnarray*}
\displaystyle N^{0}_q(f)(t)&=&\sup_{x,v}|(1+|v|)^qf(x,v)|,\cr
\displaystyle N^{k}_q(f)(t)&=&\sum_{|\alpha|=k}\sup_{x,v}|(1+|v|)^q\partial^{\alpha}_x
f(x,v)|,\cr
\displaystyle \overline{N}^{k}_q(f)(t)&=&\sum_{|\alpha|+|\beta|=k}\sup_{x,v}|(1+|v|)^q\partial^{\alpha}_x\partial^{\beta}_v
f(x,v)|,
\end{eqnarray*}
and
\begin{eqnarray*}
\displaystyle N^0_q(f^n_{R})(t)&=&\sup_{m,\ell}|(1+|v_{\ell}|)^qf^n_R(x_{m},v_{\ell})\Big|,\cr
\displaystyle N^{1}_q(f^n_{R})(t)&=&\sup_{m,\ell}\Big|(1+|v_{\ell}|)^q\frac{f^n_R(x_{m+1},v_{\ell})-f^n_R(x_{m},v_{\ell})}{\triangle x}\Big|.
\end{eqnarray*}
For simplicity, we set
\begin{eqnarray*}
\displaystyle N_q(f)(t)&=&N^0_{q}(f)(t)+N^{1}_{q}(f)(t),\cr
\displaystyle \overline{N}_q(f)(t)&=&N^0_{q}(f)(t)+\overline{N}^{1}_{q}(f)(t).
\end{eqnarray*}
\end{itemize}
\begin{remark}
Note that we have deliberately distinguished  $\overline{N}_q$ from $N_q$. This simplifies many computations
in later sections.
\end{remark}
%
%
\section{{\bf Existence and uniqueness of smooth solutions}}
%
%
%
In this section, we recall relevant existence results of \eqref{main.1}.
The existence and uniqueness was first obtained in \cite{P-P} and the regularity was investigated in \cite{Iss}.
The following theorem is a slight simplification of the corresponding results in \cite{Iss, P-P}, which is enough for our purpose.
For the proof, we refer to \cite{Iss, P-P}.
\begin{theorem}\label{Existence_Theorem}\emph{\cite{Iss, P-P}} Suppose $f_0\geq 0$, $f_0\in L^1(\mathbb{T}\times \bbr)$.
Suppose further that there exists constants $C_1$ and $C_2$ such that
\begin{align}
\begin{aligned}\label{LowerBoundforLocalDensity}
&\displaystyle\int_{\bbr^n}f_0(x-vt,v)dv\geq C_1>0,\cr
&\hspace{2cm} \overline{N}_{q}(f_0)\leq C_2.
\end{aligned}
\end{align}
for $p>5$.
Then there exists a positive number $T_f$ and a unique solution of the BGK model (\ref{main.1}) such that\newline
\indent(1) $\overline{N}_q$-norm of $f$ is uniformly bounded in $[0,T_f]$:
\begin{eqnarray*}
 \overline{N}_{q}(f)(t)\leq C_{T_f} \quad\mbox{for}~ t\in [0,T_f],
\end{eqnarray*}
\indent(2) Macroscopic fields satisfy the following estimates:
\begin{align}
\begin{aligned}\label{ULBoundsMacroscopicContinuous}
&\|\rho(t)\|_{L^{\infty}_x}+\|U(t)\|_{L^{\infty}_x}+\|T(t)\|_{L^{\infty}_x}\leq C_q N^0_q(f^0_R)e^{C_qT_f},\cr
&\hspace{2.4cm}\rho(x,t)\geq C_qe^{-C_qT_f},\cr
&\hspace{2.2cm}T(t)\geq C_{q}e^{-C_qT_f}>0.
\end{aligned}
\end{align}
\end{theorem}
In what follows, we list some of the important estimates satisfied by the smooth solutions.
Readers are referred to \cite{P-P} for the proof.
\begin{lemma}\label{control_macroscopic values_continuous}\emph{\cite{P-P}} Under the assumptions of the previous theorem,
the following estimates hold.
\begin{align}
\begin{aligned}\label{PS1}
\frac{\rho}{(T)^{\frac{d}{2}}}&\leq C_q N_0(f),\cr
\rho(T+|U|^2)^{\frac{q-d}{2}}&\leq C_q N_q(f)\quad(q>d+2),\cr
\frac{\rho}{( T^n_R+|U^n_R|^2)^{\frac{d-q}{2}}}&\leq C_qN_q(f)\quad(q<d),\cr
\frac{\rho| U|^{d+q}}{[(T+|U|^2)T]^{\frac{d}{2}}}&\leq C_qN_q(f)\quad(q>1).
\end{aligned}
\end{align}
\end{lemma}
\section{{\bf Main result}}
We are now in a place to state our main result.
%
%
%
%
\begin{theorem}\label{maintheorem}
Let $f$ be a smooth solution corresponding to a nonnegative initial datum $f_0$ satisfying the hypotheses of
theorem \ref{Existence_Theorem}. Let $f^n_R$ be the approximate solution
constructed iteratively by \eqref{refom_scheme}. Suppose  
that the time step is bounded in the sense that:
\begin{equation}\label{no.small1}
\triangle t<\max\{\frac{1}{2},~\kappa\}
\end{equation}
and the size of spatial and velocity meshes satisfies the following smallness assumption:
\begin{equation}\label{smallness1}
\triangle x+\triangle v<\frac{\int f_{0R}(x-vT_f,v)dv}{2\overline{N}_q(f_0)(2+T_f)},
\end{equation}
where
$f_{0R}=f_0\mathcal{X}_{|v|<R}.$
Then we have
\begin{eqnarray*}
\|f(\cdot,\cdot ,T_f) - f^{N_t}_R\|_{L^1_2}\leq C(T,q,f_0)
\Big(\triangle x+\triangle v+\frac{\triangle v+\triangle x}{\triangle t}+\triangle t+\frac{1}{(1+R)^{q+1}}~\Big).\\
\end{eqnarray*}
\end{theorem}
\begin{remark}
1. Note that (\ref{no.small1}) is not a smallness condition. Therefore, Theorem \ref{maintheorem} shows that our scheme works well even for large time steps.\newline
2. The condition (\ref{no.small1}), (\ref{smallness1}) on the mesh size is introduced to derive
 the lower bound estimates on discrete macroscopic fields. See lemma \ref{ULBoundsMacroscopicDiscreteLemma}.  \newline
3. For convergence, we need to set the size of $\triangle x$ and $\triangle v$ to be
 comparable with $\triangle t$. For example, if we set $\triangle x=\triangle v=(\triangle t)^{1+m}$, then we have
\begin{eqnarray*}
\|f(\cdot,\cdot ,T_f) - f^{N_t}_R\|_{L^1_2}\leq C(T,q,f_0)
\Big((\triangle t)^m+\triangle t+\frac{1}{(1+R)^{q+1}}~\Big).
\end{eqnarray*}
4. For high-order methods \cite{R-P, Sant}, we expect to obtain a high-order error estimate of the following form:
\begin{eqnarray*}
\|f(\cdot,\cdot ,T_f) - f^{N_t}_R\|_{L^1_2}\leq C
\Big((\triangle x)^{\ell}+(\triangle v)^m+\frac{(\triangle v)^{\ell}+(\triangle x)^m}{\triangle t}+(\triangle t)^n+\frac{1}{(1+R)^{q+1}}~\Big),
\end{eqnarray*}
which we leave  for the future research.
\end{remark}
\section{{\bf Basic estimates}}
In this section, we present several estimates which will be crucial in later sections. Unless it is necessary,
we do not restrict our argument to one dimensional problems 
and present the result in general $d$-dimension.
\begin{lemma}\label{Nqf_Nqtildef} The reconstruction procedure does not increase the $N_q$-norm of the discrete distribution function:
\[
N_q( \widetilde  f^n_R)\leq N_q(f^n_R).
\]
\end{lemma}
\begin{proof}
(i) The estimate of $N^0_q(\widetilde f^n_R)$: We observe from the definition of $\widetilde f^n_R$
\begin{eqnarray*}
N^0_q(\widetilde f^n_R)\
&=&\sup_{i,j}\big|\widetilde f^n_R(x_i,v_j)(1+|v_j|)^q\big|\cr
&=&\sup_{i,j}\Big|\Big(\frac{x(i,j)-x_{s(i,j),j}}{\triangle x}~f^{n}_{s(i,j)+1,j,R}
+\frac{x_{s(i,j)+1,j}-x(i,j)}{\triangle x}~f^{n}_{s(i,j),j,R}\Big)(1+|v_j|)^q\Big|~\cr
&\leq&\sup_{i,j}\big|\max\{f^{n}_{s(i,j),j,R},f^{n}_{s(i,j)+1,j,R}\}(1+|v_j|)^q\big|\cr
&\leq&\sup_{i,j}\big|f^n_{i,j,R}(1+|v_j|)^q\big|\cr
&=&N_q(f^n_R).
\end{eqnarray*}
(ii) The estimate of $N^1_q(\widetilde f^n_R)$: We first define
\begin{eqnarray*}
a=\frac{x(i+1,j)-x_{s(i+1,j),j}}{\triangle x}=\frac{x(i,j)-x_{s(i,j),j}}{\triangle x},
\end{eqnarray*}
which gives
\begin{eqnarray*}
1-a=\frac{x_{s(i+1,j)+1,j}-x(i+1,j)}{\triangle x}=\frac{x_{s(i,j)+1,j}-x(i,j)}{\triangle x}.
\end{eqnarray*}
Therefore, we have from the definition of $\widetilde{f}^n_R$
\begin{eqnarray*}
&&\Big|\frac{\widetilde{f}^n_R(x_{i+1},v_{j})-\widetilde{f}^n_R(x_i,v_{j})}{\triangle x}\Big|\cr
&&\qquad=\Big|\Big[\frac{af_{s(i+1,j)+1,j R}+(1-a)f_{s(s+1,j),jR}}{\triangle x}\Big]
-\Big[\frac{af_{s(i,j)+1,j R}+(1-a)f_{s(s,j),jR}}{\triangle x}\Big]\Big|\cr
&&\qquad=\Big|a\Big[\frac{f_{s(i+1,j)+1,j R}-f_{s(i,j)+1,j R}}{\triangle x}\Big]
+(1-a)\Big[\frac{f_{s(s+1,j),jR}-f_{s(s,j),jR}}{\triangle x}\Big]\Big|\cr
&&\qquad\leq a\Big|\frac{f_{s(i+1,j)+1,j R}-f_{s(i,j)+1,j R}}{\triangle x}\Big|+
(1-a)\Big|\frac{f_{s(s+1,j),jR}-f_{s(s,j),jR}}{\triangle x}\Big|\cr
&&\qquad \leq \frac{aN_q^1(f^n_R)}{(1+|v_j|)^q}+\frac{(1-a)N_q^1(f^n_R)}{(1+|v_j|)^q}\cr
&&\qquad \leq \frac{N_q^1(f^n_R)}{(1+|v_j|)^q}.
\end{eqnarray*}
Hence we have
\begin{eqnarray*}
N^1_q(\widetilde{f}^n_R)&=&\sup_{i,j}\Big|(1+|v_q|)^q\frac{\widetilde{f}^n_R(x_{i+1},v_{j})-\widetilde{f}^n_R(x_i,v_{j})}{\triangle x}\Big|\cr
&\leq& N_q^1(f^n_R).
\end{eqnarray*}
We then combine the above two estimates to obtain
\begin{eqnarray*}
N_q(\widetilde{f}^n_R)&=&N_q(\widetilde{f}^n_R)+N^1_q(\widetilde{f}^n_R)\cr
&\leq& N_q(f^n_R)+N^1_q(f^n_R)\cr
&=&N_q(f^n_R).
\end{eqnarray*}
\end{proof}
%
%
\begin{lemma}\label{control_macroscopic values}Suppose $N_q(f^n_R)<\infty$ with $q>d+2$. Then the following estimates hold.
\begin{align}
\begin{aligned}\label{PS1}
\frac{\rho^n_R}{(T^n_R)^{\frac{d}{2}}}&\leq C_q N_0(f^n_R),\cr
\rho^n_R(T^n_R+|U^n_R|^2)^{\frac{q-d}{2}}&\leq C_q N_q(f^n_R)\quad(q>d+2),\cr
\frac{\rho^n_R}{( T^n_R+|U^n_R|^2)^{\frac{d-q}{2}}}&\leq C_qN_q(f^n_R)\quad(q<d),\cr
\frac{\rho^n_R| U^n_R|^{d+q}}{[(T^n_R+|U^n_R|^2)T^n_R]^{\frac{d}{2}}}&\leq C_qN_q(f^n_R)\quad(q>1),
\end{aligned}
\end{align}
and
\begin{align}
\begin{aligned}\label{PS2}
\frac{\widetilde\rho^n_R}{(\widetilde T^n_R)^{\frac{d}{2}}}&\leq C_q N_0(f^n_R),\\
\widetilde\rho^n_R(\widetilde T^n_R+|\widetilde U^n_R|^2)^{\frac{q-d}{2}}&\leq C_q N_q(f^n_R)\quad(q>d+2),\\
\frac{\widetilde\rho^n_R}{(\widetilde T^n_R+|\widetilde U^n_R|^2)^{\frac{d-q}{2}}}&\leq C_qN_q(f^n_R)\quad(q<d),\\
\frac{\widetilde\rho^n_R|\widetilde U^n_R|^{d+q}}{[(\widetilde T^n_R+|\widetilde U^n_R|^2)\widetilde T^n_R]^{\frac{d}{2}}}&\leq C_qN_q(f^n_R)\quad(q>1).
\end{aligned}
\end{align}
\end{lemma}
\begin{proof}
We only prove (\ref{PS2}). (\ref{PS1}) can be proved in a similar manner.\newline
(i) The estimate of $\frac{\widetilde\rho^n_R}{(\widetilde T^n_R)^{\frac{d}{2}}}$:
We have from \eqref{Conservatives_Reform}
\begin{eqnarray*}
\widetilde\rho^n_R&=&\int \widetilde f^n_R(x,v)dv\\
&=&\frac{1}{D^2}\int_{|C_2(v)-\widetilde U^n_R(x)|\geq D} \widetilde f^n_R|C_2(v)-\widetilde U^n_R(x)|^2dv+\int_{|C_2(v)-\widetilde U^n_R|<D}\widetilde f^n_Rdv\\
&=&\frac{d}{D^2}\widetilde\rho^n_R\widetilde T^n_R+CD^dN_q(f^n_R).
\end{eqnarray*}
We take $D=\Big(\frac{d\widetilde\rho^n_R\widetilde T^n_R}{CN_q(f^n)}\Big)^{\frac{1}{d+4}}$ to obtain
\[
\frac{\widetilde\rho^n_R}{(\widetilde T^{n}_R)^{\frac{N}{2}}}\leq C_qN_0(f^n_R).
\]
This proves \eqref{PS1}.\newline
(ii) The estimate of $\widetilde\rho^n_R(\widetilde T^n_R+|\widetilde U^n_R|^2)^{\frac{q-d}{2}}$ : We note that
\begin{eqnarray*}
\widetilde\rho^n_R(d\widetilde T^n_R+|\widetilde U^n_R|^2)&=&\int \widetilde f^n_R|C_2(v)-U^n_R(x)|^2d v+\rho^n_R(x)U^n_R(x)\cr
&=&\int \widetilde f^n_R|C_2(v)|^2d v\cr
&\leq&\int_{|C_2(v)|>D}\widetilde f^n_{R}\frac{|C_2(v)|^q}{|C_2(v)|^{q-2}}d v+\int_{|C_2(v)|\leq D}\widetilde f^n_{R}|C_2(v)|^2dv\cr
&\leq&N_q(f^n_R)\int_{|C_2(v)|>D}\frac{1}{|C_2(v)|^{q-2}}dv+D^2\int_{|C_2(v)|\leq R}\widetilde f^n_{R}dv\cr
&\leq& C_q\frac{N_{q}(f^n_R)}{D^{q-d-2}}+\widetilde\rho^n_RD^2.
\end{eqnarray*}
We take
\[
D=\Big(\frac{N_q(f^n_R)}{\widetilde\rho^n_R}\Big)^{\frac{1}{q-d}}
\]
to complete the proof.\newline
(iii) The estimate of $\frac{\widetilde\rho^n_R}{(\widetilde T^n_R+|\widetilde U^n_R|^2)^{\frac{d-q}{2}}}$ : Note that
\begin{eqnarray*}
\widetilde\rho^n_R&=&\int_{\bbr^d}\widetilde f^n_Rdv\cr
&=&\int_{|C_2(v)|\leq D}\widetilde f^n_Rdv+\int_{|C_2(v)|> D}\widetilde f^n_Rdv\cr
&=&\int_{|C_2(v)|\leq D}\widetilde f^n_R\frac{|C_2(v)|^q}{|C_2(v)|^q}dv+\frac{1}{D^2}\int_{|C_2(v)|> D}\widetilde f^n_R|C_2(v)|^2dv\cr
&\leq&C_qD^{d-q}N_q(f^n_R)+\frac{\rho^n_R(dT^n_R+|U^n_R|^2)}{D^2}.
\end{eqnarray*}
We take  $D^{d-q+2}=\widetilde\rho(d\widetilde T^n_R+|\widetilde U^n_R|^2)^{1-\frac{2}{d-q+2}}$
to see
\[
\rho^n_R\leq C_qN_q(f^n_R)^{\frac{2}{d-q+2}}[\widetilde\rho(d \widetilde T^n_R+|\widetilde U^n_R|^2)]^{1-\frac{2}{d-q+2}}.
\]
(iv) The estimate of $\frac{\widetilde\rho^n_R|\widetilde U^n_R|^{d+q}}{[(\widetilde T^n_R+|\widetilde U^n_R|^2)]^{\frac{d}{2}}}$ :
For $q>1$, we have by H\"{o}lder inequality,
\begin{eqnarray*}
\widetilde\rho^n_R|\widetilde U^n_R|&\leq&\int \widetilde f^n_R|C_2(v)|dv\cr
&\leq&\int_{|C_2(v)-\widetilde U^n_R|\leq D} \widetilde f^n_R|C_2(v)|dv+\int_{|C_2(v)-\widetilde U^n_R|>D} \widetilde f^n_R|C_2(v)|dv\cr
&\leq&\big(\widetilde\rho^n_R\big)^{1-\frac{1}{q}}\Big(\int_{|C_2(v)-\widetilde U^n_R|\leq D} \widetilde f^n_R|C_2(v)|^qdv\Big)^{\frac{1}{q}}
+\frac{1}{D}\int_{\bbr^d}|C_2(v)-\widetilde U^n_R||C_2(v)|\widetilde f^n_Rdv\cr
&\leq&C(\widetilde\rho^n_R)^{1-\frac{1}{q}}N_q(f^n_R)^{\frac{1}{q}}D^{\frac{D}{q}}
+\frac{1}{D}\Big(\int|C_2(v)|^2\widetilde f^n_Rdv\Big)^{\frac{1}{2}}\Big(\int|C_2(v)-\widetilde U^n_R|^2\widetilde f^n_Rdv\Big)^{\frac{1}{2}}\cr
&\leq&C(\widetilde\rho^n_R)^{1-\frac{1}{q}}N_q(f^n_R)^{\frac{1}{q}}D^{\frac{d}{q}}+\frac{\widetilde\rho^n_R}{D}(d\widetilde T^n_R+|\widetilde U^n_R|^2)^{\frac{1}{2}}\widetilde T_R^{n\frac{1}{2}}.
\end{eqnarray*}
We maximize the estimate by taking
\[
D^{\frac{d+q}{q}}=\frac{\widetilde\rho^{\frac{1}{q}}(d\widetilde T^n_R+|\widetilde U^n_R|^2)^{\frac{1}{2}}\widetilde T^{\frac{1}{2}}}{N_q(f^n_R)^{\frac{1}{q}}}
\]
to obtain the desired result.
\end{proof}
The next lemma shows that the $N_q$- norm of the discrete local Maxwellian can be controlled by the $N_q$- norm 
of the approximate distribution function.
%
%
\begin{lemma}\label{ControlofMaxwellian} Suppose $N_q(f^n_R)<\infty$ with $q>d+2$, then we have
\[
N_q(\mathcal{M}^n(f^n_R))\leq C_qN_q(f^n_R).
\]
\end{lemma}
\begin{proof}
(I) The estimate of $N^0_q(\mathcal{M}^n(f^n_R))$ :
We observe that
\begin{align}
\begin{aligned}\label{DecompositionM}
(1+|C_2(v)|)^q {\mathcal M}^n(f^n_R)&\leq C_q(1+|U^n_R|^q+|C_2(v)-U^n_R|^q) {\mathcal M}^n(f^n_R)\cr
&\leq C_q\Big(\frac{\rho^n_R}{(T^n_R){\frac{d}{2}}}+\rho^n_R\frac{|U^n_R|^q}{(T^n_R)^{\frac{d}{2}}}+\rho^n_R (T^n_R)^{\frac{(q-d)}{2}}\Big).
\end{aligned}
\end{align}
By \eqref{PS1}, we have
\begin{equation}
\frac{\rho^n_R}{(T^n_R){\frac{d}{2}}}, \hspace{0.15cm}\rho^n_R (T^n_R)^{\frac{(q-d)}{2}}\leq C_q N_q(f^n_R).\label{rhoT}
\end{equation}
The estimate of $\rho^n_R\frac{|U^n_R|^q}{(T^n_R)^{\frac{d}{2}}}$ is more involved. We divide it into
the following two cases.\\
Case 1: $|U^n_R|>(T^n_R)^{\frac{1}{2}}$.\\
We have from (\ref{PS1})
\begin{align}
\begin{aligned}\label{rhoUT1}
\rho^n_R\frac{|U^n_R|^q}{(T^n_R)^{\frac{d}{2}}}&\leq\rho^n_R\frac{|U^n_R|^{d+q}}{|U^n_R|^d(T^n_R)^{\frac{d}{2}}}\cr
&\leq C_q\frac{\rho^n_R|U^n_R|^{d+q}}{(T^n_R+|U^n_R|^2)^{\frac{d}{2}}(T^n_R)^{\frac{d}{2}}}\cr
&\leq C_qN_q(f^n_R).
\end{aligned}
\end{align}
Case 2: $|U^n_R|\leq (T^n_R)^{\frac{1}{2}}$\\
We apply \eqref{PS1} to see
\begin{align}
\begin{aligned}\label{rhoUT2}
\rho^n_R\frac{|U^n_R|^q}{(T^n_R)^{\frac{d}{2}}}&\leq\rho^n_R (U^n_R)^{q-d}\cr
&\leq\rho^n_R (T^n_R+|U^n_R|^2)^{\frac{q-d}{2}}\cr
&\leq C_qN_q(f^n_R).
\end{aligned}
\end{align}
We combine the Case I and Case II to obtain
\begin{equation}
\rho^n_R\frac{|U^n_R|^q}{(T^n_R)^{\frac{d}{2}}}\leq C_qN^0_q(f^n_R).\label{rhoUTfinal}
\end{equation}
We then substitute (\ref{rhoT}), (\ref{rhoUTfinal}) into (\ref{DecompositionM}) to complete the proof.\newline

(II) The estimate of $N^1_q(\mathcal{M}(f^n_R))$ : 
We have by Taylor's theorem
\begin{align}
\begin{aligned}\label{MfnR-MfnR}
&{\mathcal M}^n(f^n_R)(x_{i+1},v_j)-{\mathcal M}^n(f^n_R)(x_i,v_j)\cr
&\hspace{0.5cm}=\frac{\rho^n_{i+1,R}}{\sqrt{(2\pi T^n_{i+1,R})^d}}\exp\Big(-\frac{|v_j-U^n_{i+1,R}|^2}{2T^n_{i+1,R}}\Big)
-\frac{\rho^n_{i,R}}{\sqrt{(2\pi T^n_{i,R})^d}}\exp\Big(-\frac{|v_j-U^n_{i,R}|^2}{2T^n_{i,R}}\Big)\cr
&\hspace{0.5cm}=(\rho^n_{i+1,R}-\rho^n_{i,R})\frac{\partial{\mathcal M}^n(\theta)}{\partial \rho}
+(U^n_{i+1,R}-U^n_{i,R})\cdot\frac{\partial {\mathcal M}^n(\theta)}{\partial U_i}\cr
&\hspace{0.5cm}+(T^n_{i+1,R}-T^n_{i,R})\frac{\partial {\mathcal M}^n(\theta)}{\partial T},
\end{aligned}
\end{align}
where
\[
\frac{\partial {\mathcal M}^n(\theta)}{\partial X}\equiv
\frac{\partial {\mathcal M}^n}{\partial X}\Big|_{\theta f^n_{i+1,j}+(1-\theta)f^n_{i,j}}
\]
for some $0\leq\theta\leq1$.
We recall that
\begin{eqnarray*}
\frac{\partial {\mathcal M}^n}{\partial \rho^n_R}&=&\frac{1}{\sqrt{(2\pi T^n_R)^d}}\exp\Big(-\frac{|v_j-U^n_R|^2}{2T^n_R}\Big)
\end{eqnarray*}
to get
\begin{eqnarray*}
\frac{\partial {\mathcal M}^n}{\partial \rho^n_R}(1+|v_j|)^q&=&\frac{1}{\sqrt{(2\pi T^n_R)^d}}\exp\Big(-\frac{|v_j-U^n_R|^2}{2T^n_R}\Big)(1+|v_j|)^q\\
&\leq&\frac{1}{\sqrt{(2\pi T^n_R)^d}}\exp\Big(-\frac{|v_j-U^n_R|^2}{2T}\Big)(1+|v_j-U^n_R|^q+|U^n_R|^q)\\
&\leq& C_{q,T}.
\end{eqnarray*}
Other estimates can be obtained similarly as follows:
\begin{eqnarray*}
\frac{\partial {\mathcal M}^n}{\partial U^n_R}(1+|v_j|)^q&\leq& C_{q,T},\\
\frac{\partial {\mathcal M}^n}{\partial T^n_R}(1+|v_j|)^q&\leq& C_{q,T}.
\end{eqnarray*}
We substitute the above estimates into \eqref{MfnR-MfnR} to see
\begin{align}
\begin{aligned}\label{|MfnR-MfnR|}
&|{\mathcal M}^n(f^n_R)(x_{i+1},v_j)-{\mathcal M}^n(f^n_R)(x_i,v_j)|(1+|v_j|)^q\\
&\hspace{2cm}\leq C_q\Big(|\rho^n_{i+1,R}-\rho^n_{i,R}|
+|U^n_{i+1,R}-Un_{i,R}|+|T^n_{i+1,R}-T^n_{i,R}|\Big).
\end{aligned}
\end{align}
We now estimate each terms separately. First we observe
\begin{align}
\begin{aligned}\label{|rhonR-rhonR|}
|\rho^n_{i+1,R}-\rho^n_{i,R}|&=\Big|\sum_{j} (f^n_{i+1,j,R}-f^n_{i,j,R})\triangle v\Big|\cr
&=\sup_{i,j}\Big|\frac{f^n_{i+1,j,R}-f^n_{i,j,R}}{\triangle x}(1+|v_j|)^{q}\Big|\sum_{j}\frac{\triangle x \triangle v}{(1+|v_j|)^{q}}\cr
&\leq C_qN^1_{q}(f^n_R)\triangle x.
\end{aligned}
\end{align}
Similarly, we have
\begin{eqnarray}
|U^n_{i+1,R}-U^n_{i,R}|&\leq&C_qN^1_{q}(f^n_R)\triangle x,\label{|UnR-UnR|}\\
|T^n_{i+1,R}-T^n_{i,R}|&\leq&C_qN^1_{q}(f^n_R)\triangle x.\label{|TnR-TnR|}
\end{eqnarray}
We substitute \eqref{|rhonR-rhonR|}, \eqref{|UnR-UnR|}, \eqref{|TnR-TnR|} into \eqref{|MfnR-MfnR|} to obtain
\begin{eqnarray*}
N^1({\mathcal M}^n(f^n_R))&=&\sup_{i,j}\Big|\frac{{\mathcal M}^n(f^n_R)(x_{i+1},v_j)-{\mathcal M}^n(f^n_R)(x_i,v_j)}{\triangle x}(1+|v_j|)^2\Big|\\
&\leq&C_qN^1_{q}(f^n_R).
\end{eqnarray*}
\end{proof}
We now establish the stability estimate for the scheme (\ref{refom_scheme}).
%
%
\begin{lemma}\label{ControllonFnR}
Suppose $N_q(f^0_R)<\infty$ with $q>d+2$. Then we have
\[
N_q(f^n_R)<e^{C_qT_f}N_q(f^0_R).
\]
\end{lemma}
\begin{proof}
We take $N_q$ norm on both sides of (\ref{refom_scheme}) and apply Lemma \ref{Nqf_Nqtildef} and Lemma \ref{ControlofMaxwellian} to see
\begin{eqnarray*}
N_q(f^{n}_R)&\leq&\frac{\kappa}{\kappa+\triangle t}N_q(f^{n-1}_R)
+\frac{\triangle t}{\kappa+\triangle t}N_q({\mathcal M}(\widetilde{f}^{n-1}_R))\\
&\leq&\frac{\kappa}{\kappa+\triangle t}N_q(f^{n-1}_R)
+\frac{\triangle t}{\kappa+\triangle t}C_TN_q(f^{n-1}_R)\\
&\leq&\frac{\kappa+C_q\triangle t }{\kappa+\triangle t}N_q(f^{n-1}_R)\\
&\leq&\Big(1+\frac{(C_q-1)\triangle t }{\kappa+\triangle t}\Big)N_q(f^{n-1}_R).
\end{eqnarray*}
Iterating the above inequality, we obtain
\begin{eqnarray*}
N_q(f^{n}_R)&\leq& \Big(1+\frac{(C_q-1)\triangle t }{\kappa+\triangle t}\Big)^{n}N_q(f^{0}_R)\cr
&\leq&e^{\frac{(C_q-1)N\triangle t}{\kappa+\triangle t}}(f^{0}_R)\cr
&\leq&e^{\frac{(C_q-1)T}{\kappa+\triangle t}}N_q(f^{0}_R),
\end{eqnarray*}
where we used
$(1+x)^n\leq e^{nx}$ and $n\triangle t\leq N_t\triangle t=T_f$.
\end{proof}
%
%
\begin{lemma}\label{ULBoundsMacroscopicDiscreteLemma}
Let $q>d+2$ and $N_q(f^n_R)<\infty$. Suppose that the time step is bounded in the sense that
\begin{equation}\label{no.small2}
\triangle t<\max\{\frac{1}{2},\kappa\}
\end{equation}
and the mesh size for spacial and velocity nodes satisfies the following smallness condition:
\begin{equation}\label{smallness2}
\triangle x+\triangle v<\frac{\int f_{0R}(x-vT_f,v)dv}{2\overline{N}_q(f_0)(2+T_f)}.
\end{equation}
Then the following estimates holds for approximate macroscopic fields.
\begin{eqnarray}
&\|\rho^n_R(t)\|_{L^{\infty}_x}+\|U^n_R(t)\|_{L^{\infty}_x}+\|T^n_R(t)\|_{L^{\infty}_x}\leq C_qN_q(f^0_R)e^{C_qT_f},\nonumber\\
&\rho^n_R(x,t)\geq C_qe^{-C_qT_f},\label{ULBoundsMacroscopicDiscrete}\\
&T^n_R(x,t)\geq C_{q}e^{-C_qT_f}>0.\nonumber
\end{eqnarray}
\end{lemma}
\begin{proof}
Note that we have from Lemma \ref{ControllonFnR}
\begin{eqnarray*}
\rho^n_R(t)&=&\int f^n_R(x,v)dv\\
&\leq&N_q(f^n_R)\sum_{j}\frac{\triangle v}{(1+|v_{j}|)^q}\\
&\leq&C_qe^{C_qT_f}N_q(f^0_R).
\end{eqnarray*}
To proceed to the estimates for $U^n_R$ and $T^n_R$, we need to establish the lower bound for $\rho^n_R(x,t)$ and $T^n_R(x,t)$ first.
Note that we have from \eqref{refom_scheme}
\begin{align}
\begin{aligned}\label{rho_geq_Ef1}
\rho^n_R(x)
&=\int_{\bbr^d} f^n_R(x,v) dv\cr
&\geq\frac{\kappa}{\kappa+\triangle t}\int \widetilde{f}^{n-1}_R(x,v)dv\cr
&=\frac{\kappa}{\kappa+\triangle t}E_x\Big(\int f^{n-1}_R(x_i-C_2(v)\triangle t,v) dv\Big).
\end{aligned}
\end{align}
In the last line, we used
\begin{eqnarray*}
&&\hspace{-0.5cm}\int \widetilde{f}^{n-1}_R(x,v)dv\cr
&&=\int\sum_{i,j}\Big(\frac{x-x_{i}}{\triangle x}\widetilde f^{n}_{i+1,j}
+\frac{x_{i+1}-x}{\triangle x}\widetilde f^n_{i,j}\Big){{\mathcal A}^R_{i,j}}(x,v)dv\cr
&&=\sum_{i}\Big\{(\frac{x-x_{i}}{\triangle x}\Big(\sum_j\int\widetilde f^{n}_{i+1,j}{{\mathcal X}_{A_{v_j}}}dv\Big)
+\frac{x_{i+1}-x}{\triangle x}\Big(\sum_j\int\widetilde f^n_{i,j}{{\mathcal X}_{A_{v_j}}}(x,v)dv\Big)\Big\}
{{\mathcal A}^R_{x_i}}(x,v)\cr
&&=\sum_{i}\Big\{(\frac{x-x_{i}}{\triangle x}\Big(\sum_j\widetilde f^{n}_{i+1,j}\triangle v\Big)
+\frac{x_{i+1}-x}{\triangle x}\Big(\sum_j\widetilde f^n_{i,j}\triangle v\Big)\Big\}
{{\mathcal A}^R_{x_i}}(x,v)\cr
&&\equiv E_x\Big(\sum_j\widetilde f^{n-1}_{i,j,R}\triangle v\Big)\cr
&&=E_x\Big(\sum_j\Big(\frac{x_i-v_j\triangle t-x_s}{\triangle x}f^{n-1}_{s+1,j,R}+\frac{x_{s+1}-(x_i-v_j\triangle t)}{\triangle x} f^{n-1}_{s,j}\Big)\triangle v\Big)\cr
&&=E_x\Big(\int\sum_j\Big(\frac{x_i-C_2(v)\triangle t-x_s}{\triangle x}f^{n-1}_{s+1,j,R}+\frac{x_{s+1}-(x_i-C_2(v)\triangle t)}{\triangle x} f^{n-1}_{s,j}\Big)
\mathcal{A}^R_{v_j}(v)dv\Big)\cr
&&=E_x\Big(\int f^{n-1}_R(x_i-C_2(v)\triangle t,v) dv\Big).
\end{eqnarray*}
On the other hand, we note from the definition of $E_x$ that
\begin{eqnarray*}
E_x\Big(\int f^{n-1}_R(x_i-C_2(v)\triangle t,v)dv\Big)
\geq \inf_i\int f^n_R(x_i-C_2(v)\triangle t,v,t)dv.
\end{eqnarray*}
This gives from (\ref{rho_geq_Ef1})
\[
\rho^{n}_R(x)\geq \frac{\kappa}{\kappa+\triangle t}\inf_i\int f^n_R(x_i-C_2(v)\triangle t,v,t)dv.
\]
We iterate the above lower bound estimate to obtain
\begin{align}
\begin{aligned}\label{rho_R}
\rho^{n}_R(x)&\geq\Big(\frac{\kappa}{\kappa+\triangle t}\Big)^n\inf_i\int f^0_R(x_i-C_2(v)n\triangle t,v)dv\cr
&=\Big(\frac{ \kappa}{\kappa+\triangle t}\Big)^n\inf_i\int f^0_R(x_i-C_2(v)T_f,v)dv.
\end{aligned}
\end{align}
Then we employ the following elementary inequality
\[
(1-x)^n\geq e^{-8nx}\quad(0<x<1-\varepsilon)~\mbox{ for } 0<\varepsilon<1,
\]
to see
\begin{eqnarray}
\Big(\frac{\kappa}{\kappa+\triangle t}\Big)^n
&=&\Big(1-\frac{\triangle t}{\kappa+\triangle t}\Big)^n\nonumber\\
&\geq&e^{-\frac{8n\triangle t}{\kappa+\triangle t}}\nonumber\\
&\geq&e^{-\frac{8T_f}{\kappa+\triangle t}}.\nonumber
\end{eqnarray}
Here we used the fact that the boundedness asuumption on $\triangle t$ implies
\[
\frac{\triangle t}{\kappa+\triangle t}<1-\min\big\{\frac{1}{2},\kappa\big\}.
\]
Hence (\ref{rho_R}) gives
\begin{eqnarray}
\rho^{n}_R(x)\geq e^{-\frac{8T_f}{\kappa+\triangle t}}\inf_i\int f^0_R(x_i-C_2(v)T_f,v)dv.\label{e-TSn(rho)}
\end{eqnarray}
Now we should replace the estimate of $f^0_{R}$ with the estimate of $f_{0R}$.
We first observe that the difference between $f^0_R$ and $f_{0R}$ can be estimated as follows:
\begin{eqnarray*}
&&f_{0R}(x_i-v_j T_f,v_j)\cr
&&\hspace{1cm}\leq \frac{x_i-v_jT_f-x_s}{\triangle x}f_{0R}(x_{s+1},v_j)
+\frac{x_{s+1}-(x_i-v_jT_f)}{\triangle x}f_{0R}(x_s,v_j)+\Big|\frac{\partial f_0}{\partial x}(x_{\theta},v_j)\Big|\triangle x\cr
&&\hspace{1cm}= \frac{x_i-v_jT_f-x_s}{\triangle x}f^0_{s+1,j,R}
+\frac{x_{s+1}-(x_i-v_jT_f)}{\triangle x}f^0_{s,j,R}+\Big|\frac{\partial f_0}{\partial x}(x_{\theta},v_j)\Big|\triangle x\cr
&&\hspace{1cm}=f^0_R(x_i-v_jT_f,v_j)+\Big|\frac{\partial f_0}{\partial x}(x_{\theta},v_j)\Big|\triangle x\cr
&&\hspace{1cm}\leq f^0_R(x_i-v_jT_f,v_j)+N_q(f_0)\frac{\triangle x}{(1+|v_{j-\frac{1}{2}}|)^q},
\end{eqnarray*}
where $x_{\theta}\in [x_s, ~x_{s+1})$ and $s=s(i,j)$ denotes the index of spatial mesh grid such that $x_i-v_jT_f\in [x_s, ~x_{s+1})$.
We then sum over $i$ to obtain
\begin{align}
\begin{aligned}\label{ApproximateAgain}
&\sum_j f_{0R}(x_i-v_j T_f,v_j)\triangle v\cr
&\hspace{0.5cm}\leq\sum_jf^0_R(x_i-v_jT_f,v_j)\triangle v+\sum_jN_q(f_0)\frac{\triangle x\triangle v}{(1+|v_{j-\frac{1}{2}}|)^q}\cr
&\hspace{0.5cm}=\int f^0_R(x_i-C(v)T_f,v)dv+N_q(f_0)\triangle x.
\end{aligned}
\end{align}
On the other hand, application of Taylor's theorem on each interval $[v_{j-\frac{1}{2}},v_{j+\frac{1}{2}})$ gives
\begin{align}
\begin{aligned}\label{RiemannSum}
&\Big|\int f_{0R}(x_i-vT_f,v)dv- \sum_j f_{0R}(x_i-v_j T_f,v_j)\triangle v~\Big|\cr
&\hspace{2cm}\leq (1+T_f)(\triangle x+\triangle v)\overline{N}^{1}_q(f_{0R})\cr
&\hspace{2cm}\leq (1+T_f)(\triangle x+\triangle v)\overline{N}_q(f_{0}).
\end{aligned}
\end{align}
Substituting (\ref{ApproximateAgain}) and (\ref{RiemannSum}) into (\ref{e-TSn(rho)}), we obtain
\begin{eqnarray*}
\rho^{n}_R(x)&\geq&e^{-\frac{8T_f}{\kappa+\triangle t}}\inf_i\int f^0_R(x_i-C_2(v)T_f,v)dv\cr
&\geq& e^{-\frac{8T_f}{\kappa+\triangle t}}\inf_x\Big(\int f_{0R}(x_i-vT_f,v)dv-C\overline{N}_q(f_0)(2+T_f)(\triangle x+\triangle v)\Big).\cr
&\geq&\frac{1}{2}e^{-\frac{8T_f}{\kappa+\triangle t}}\inf_x\int f_{0R}(x_i-vT_f,v)dv.
\end{eqnarray*}
In the last line, we used the smallness assumption (\ref{smallness2}) to see
\[
\overline{N}_q(f_0)(2+T_f)(\triangle x+\triangle v)<\frac{1}{2}\int f_{0R}(x-vT_f,v)dv.
\]
Hence we have from (\ref{LowerBoundforLocalDensity})
\begin{eqnarray*}
\rho^{n}_R(x)\geq e^{-C_qT_f}C_{T_f,f_0}.
\end{eqnarray*}
We now turn to the proof of the lower bound of $T^n_R$. Note from the estimate \eqref{PS1} of Lemma \ref{control_macroscopic values}, we have
\begin{eqnarray*}
T^n_R&\geq&\Big(C_q\frac{N_q(f^n_R)}{\rho^n_R}\Big)^{\frac{2}{d}}\cr
&\geq&\Big(C_q\frac{N_q(f^0_R)}{e^{C_qT}}\Big)^{\frac{2}{d}}\cr
&=&C_q\big(N_q(f^0_R)\big)^{\frac{2}{d}}e^{-C_qT_f}.
\end{eqnarray*}
Now the pointwise upper bound estimate of $U$ and $T$ follows directly from the following observations:
\begin{eqnarray*}
U^n_R(x)&\leq&\frac{\int f^n_R(x,v)dv}{\rho^n_R(x)},\cr
T^n_R(x)&\leq&\frac{\int f^n_R(x,v)|v-U^n_R(x)|^2dv}{\rho^n_R(x)}.
\end{eqnarray*}
\end{proof}
The following continuity property of local Maxwellians is from \cite{P-P}.
%
%
\begin{lemma}\label{ContinuityofMaxwellian}\cite{P-P} Let $f$ and $g$ be solutions of (\ref{main.1})
such that $N_p(f)<\infty$ and $N_p(g)<\infty$.
Suppose $f$ and $g$ satisfy estimates \eqref{ULBoundsMacroscopicContinuous}. Then we have
\begin{equation}
\|\mathcal{M}(f)-\mathcal{M}(g)\|_{L^1_2}\leq C_{T_f}\|f - g\|_{L^1_2}.\label{Stability1}
\end{equation}
\end{lemma}
The proof of this lemma is given in \cite{P-P}, we present here the detailed
proof for the reader's convenience.
\begin{proof}
We observe from Taylor's theorem
\begin{eqnarray*}
{\mathcal M}(f)-{\mathcal M}(g)&=&{\mathcal M}(\rho_f,U_f,T_f)-{\mathcal M}(\rho_g,U_g,T_g)\\
&=&(\rho_{f-g},U_{f-g},T_{f-g})\cdot\nabla {\mathcal M}(\rho,U,T)_{\theta f+(1-\theta)g}\\
&=&\rho_{f-g}\frac{\partial {\mathcal M}(\theta)}{\partial \rho}
+U_{f-g}\cdot\frac{\partial {\mathcal M}(\theta)}{\partial U_i}
+T_{f-g}\frac{\partial {\mathcal M}(\theta)}{\partial T},
\end{eqnarray*}
where
\[
\frac{\partial {\mathcal M}(\theta)}{\partial X}\equiv \frac{\partial {\mathcal M}}{\partial X}\Big|_{\theta f+(1-\theta)g}.
\]
Multiplying $(1+|v|^2)$ to both sides and integrating with respect to $(x,v)$, we get
\begin{align}
\begin{aligned}\label{LipMac1}
&\int_{\bbr^d\times\bbt^d}|{\mathcal M}(f)-{\mathcal M}(g)|(1+|v|^2)~ dv dx\cr
&\hspace{2cm}=\int_{\bbt^d}\rho_{f-g}\int_{\bbr^d}\Big|\frac{\partial {\mathcal M}(\theta)}{\partial \rho}\Big|(1+|v|^2)dv dx\cr
&\hspace{2cm}+\int_{\bbt^d}U_{f-g}\cdot\int_{\bbr^d}\Big|\frac{\partial {\mathcal M}(\theta)}{\partial U_i}\Big|(1+|v|^2)dv dx\cr
&\hspace{2cm}+\int_{\bbt^d}T_{f-g}\int_{\bbr^d}\Big|\frac{\partial {\mathcal M}(\theta)}{\partial T}\Big|(1+|v|^2)dv dx.\cr
\end{aligned}
\end{align}
We then substitute the following estimates into (\ref{LipMac1})
\begin{eqnarray*}
\frac{\partial {\mathcal M}}{\partial \rho}&=&\frac{1}{\sqrt{(2\pi T)^N}}e^{\frac{-|v-U|^2}{2T}},\\
\frac{\partial {\mathcal M}}{\partial U_i}&=&\frac{\rho}{\sqrt{(2\pi T)^N}}\big(\frac{v_i-U_i}{T}\big)e^{\frac{-|v-U|^2}{2T}},\\
\frac{\partial {\mathcal M}}{\partial T}&=&-\frac{d}{2}\frac{1}{T}\frac{\rho}{\sqrt{(2\pi T)^{d}}}e^{\frac{-|v-U|^2}{2T}}
+\frac{\rho}{\sqrt{(2\pi T)^d}}\frac{|v-U|^2}{2T^2}e^{\frac{-|v-U|^2}{2T}}
\end{eqnarray*}
to obtain
\begin{eqnarray*}
\int_{\bbr^3}\Big|\frac{\partial {\mathcal M}(\theta)}{\partial \rho}\Big|(1+|v|^2)dv
&=& \frac{1}{\rho_{\theta}}\int_{\bbr^d}\frac{\rho_{\theta}}{\sqrt{(2\pi T_{\theta})^N}}e^{\frac{-|v-U_{\theta}|^2}{2T_{\theta}}}(1+|v|^2)dv\\
&=&\frac{1}{\rho_{\theta}}\int_{\bbr^d}\frac{\rho_{\theta}}{\sqrt{(2\pi T_{\theta})^N}}e^{\frac{-|v-U_{\theta}|^2}{2T_{\theta}}}dv\\
&&+\frac{1}{\rho_{\theta}}\int_{\bbr^d}\frac{\rho_{\theta}}{\sqrt{(2\pi T_{\theta})^N}}e^{\frac{-|v-U_{\theta}|^2}{2T_{\theta}}}|v|^2dv\\
&=&\frac{1}{\rho_{\theta}}\rho_{\theta}+\frac{1}{\rho_{\theta}}(\rho_{\theta}|U_{\theta}|^2+d\rho_{\theta}T_{\theta})\\
&=&1+|U_{\theta}|^2+dT_{\theta}.
\end{eqnarray*}
In a similar manner, we can estimate the remaining terms of (\ref{LipMac1}) as follows:
\begin{eqnarray*}
\int_{\bbr^d}\Big|\frac{\partial {\mathcal M}(\theta)}{\partial U}\Big|(1+|v|^2)dv&\leq&
\frac{\rho_{\theta}}{\sqrt{T_{\theta}}}(1+|U_{\theta}|^2+T_{\theta}),\\
\int_{\bbr^d}\Big|\frac{\partial {\mathcal M}(\theta)}{\partial T}\Big|(1+|v|^2)dv&\leq& \frac{\rho_{\theta}}{T_{\theta}}(1+|U_{\theta}|^2+T_{\theta}).
\end{eqnarray*}
Substituting these estimates into (\ref{LipMac1}), we get
\begin{align}
\begin{aligned}\label{LipMac2}
&\int_{\bbr^d}|{\mathcal M}(f)-{\mathcal M}(g)|(1+|v|^2)~ dv\cr
&\hspace{1cm}= (1+|U_{\theta}|^2+T_{\theta})\rho_{f-g}
+ \frac{\rho_{\theta}}{\sqrt{T_{\theta}}}(1+|U_{\theta}|^2+T_{\theta})|U_{f-g}|\cr
&\hspace{1cm}+\frac{\rho_{\theta}}{T_{\theta}}(1+|U_{\theta}|^2+T_{\theta})T_{f-g}.
\end{aligned}
\end{align}
We then integrate over $\bbt_x$ and employ Lemma \ref{control_macroscopic values} to have
\begin{eqnarray}
&&\|{\mathcal M}(f)-{\mathcal M}(g)\|_{L^1_2}\leq C\int_{\bbt^3}(\rho_{f-g}+|U_{f-g}|+T_{f-g})dx,\label{LipMac3}
\end{eqnarray}
where $C=C(N_q(f_0),e^{a t})$.
We estimate these terms separately.
Note first that we have
\begin{equation}
\int_{\bbr^d\times\bbt^d}\rho_{f-g}dx\leq\int_{\bbr^3\times\bbt^3}|f-g|dv dx\leq \|f-g\|_{L^1_2}.\label{rho_fg}
\end{equation}
On the other hand, we observe that
\begin{eqnarray*}
\int_{\bbt^d}|U_{f-g}|dx
&=&C(t)\int_{\bbt^3}|U_f-U_g||\rho_f|dx\quad(\because~\rho>Ce^{-Ct})\\
&=&C_T\int_{\bbt^d}|U_f\rho_f-U_g\rho_f|dx\\
&=&C_T\int_{\bbt^d}|U_f\rho_f-U_g\rho_g+U_g\rho_g-U_g\rho_f|dx\\
&=&C_T\int_{\bbt^d}|U_f\rho_f-U_g\rho_g|dx+C_T\int_{\bbt^d}|U_g||\rho_g-\rho_f|dx.
\end{eqnarray*}
The first term can be estimated as follows
\begin{eqnarray*}
\int_{\bbt^d}|U_f\rho_f-U_g\rho_g|dx&=&\int_{\bbt^d}\Big|\int_{\bbr^3}fv dv-\int_{\bbr^d}g v dv\Big| dx\\
&\leq&\int_{\bbt^d}\int_{\bbr^d}|f-g| |v| dv dx\\
&\leq& \|f-g\|_{L^1_2}.
\end{eqnarray*}
Similarly, we have
\[
\int_{\bbt^d}|U_g||\rho_g-\rho_f|dx\leq C(t)\int_{\bbr^d\times\bbt^d}|f-g|(1+|v|^2)dv dx.
\]
Combining these estimates, we get
\begin{equation}
\int_{\bbt^d}|U_{f-g}|dx\leq C_T\|f-g\|_{L^1_2}.\label{U_fg}
\end{equation}
Finally, we consider
\begin{eqnarray*}
\int_{\bbt^d}|T_{f-g}|dx
&=&C(t)\int_{\bbt^d}|T_f-T_g||\rho_f |dx\quad(\because~\rho>Ce^{-t})\\
&=&\int_{\bbt^d}|T_f\rho_f-T_g\rho_f|dx\\
&=&\int_{\bbt^d}|T_f\rho_f-T_g\rho_g+T_g\rho_g-T_g\rho_f|dx\\
&=&\int_{\bbt^d}|T_f\rho_f-T_g\rho_g|dx+\int_{\bbt^d}|T_g||\rho_g-\rho_f|dx.
\end{eqnarray*}
The first term can be estimated as follows
\begin{eqnarray*}
\int_{\bbt^3}|T_f\rho_f-T_g\rho_g|dx&=&\frac{1}{3}\int_{\bbt^d}\Big|\int_{\bbr^d}f|v|^2 dv-\int_{\bbr^d}g |v|^2 dv\Big| dx\\
&\leq&\frac{1}{3}\int_{\bbr^d\times\bbt^d}|f-g| |v| dv dx\\
&\leq& C\int_{\bbr^d\times\bbt^d}|f-g|(1+|v|^2)dv dx.
\end{eqnarray*}
Similarly, we have
\[
\int_{\bbt^d}|T_g||\rho_g-\rho_f|dx\leq C(t)\int_{\bbr^d\times\bbt^d}|f-g|(1+|v|^2)dv dx.
\]
The combination of these two estimates yields
\begin{eqnarray}
\int_{\bbt^d}|T_{f-g}|dx \leq C_T\|f-g\|_{L^1_2}.\label{T_fg}
\end{eqnarray}
We now substitute \eqref{rho_fg}, \eqref{U_fg}, \eqref{T_fg} into \eqref{LipMac3} to obtain
\[
\|{\mathcal M}(f)-{\mathcal M}(g)\|_{L^1_2}\leq C_T\|f_0-g_0\|_{L^1_2}.
\]
\end{proof}
The above continuity property also holds for discrete distribution functions.
%
%
\begin{lemma}\label{ContinuityofMaxwellianDiscrete}Let $f^n_R$ and $g^n_R$ numerical solutions
defined recursively by \ref{refom_scheme} corresponding to initial discretization $f^0_R$ and $g^0_R$ respectively.
Assume that $N_q(f^0_R)<\infty$
and $N_q(g^0_R)<\infty$ with $q>4$. Then we have
\begin{equation}
\|\mathcal{M}(f^n_R)-\mathcal{M}(g^n_R)\|_{L^1_2}\leq C_{T_f}\|f^n_R - g^n_R\|_{L^1_2}.\label{Stability1}
\end{equation}
\end{lemma}
\begin{proof}
Combination of Lemma \ref{ControllonFnR}, Lemma \ref{ULBoundsMacroscopicDiscreteLemma} and Lemma \ref{ContinuityofMaxwellian}
gives the desired result.
\end{proof}
%
%
%
%
\section{{\bf Consistent form}}
In this section, we transform the BGK equation (\ref{main.1}) to derive a consistent form which is compatible with
the reformulated scheme (\ref{refom_scheme}).
To be consistent with the notation for discrete problems, we employ the following notation for the smooth
distrbution function:
\[
\widetilde f(x,v,t)=f(x-v\triangle t,v,t).
\]
\begin{theorem}
Let $f$ be a smooth solution of \eqref{main.1}, then we have
\begin{align}
\begin{aligned}\label{Consist3}
f(x,v,t+\triangle t)&=\frac{\kappa}{\kappa+\triangle t}\widetilde{f}(x,v,t)
+\frac{\triangle t}{\kappa+\triangle t}{\mathcal M}(\widetilde{f})(x,v,s)\cr
&+\frac{\triangle t}{\kappa+\triangle t}\mathcal{R}_1+\frac{1}{\kappa+\triangle t}\mathcal{R}_2.
\end{aligned}
\end{align}
where
\begin{eqnarray*}
\mathcal{R}_1&=&{\mathcal M}(\widetilde{f})(x,v,s)-\widetilde{\mathcal M}(f)(x,v,s),\nonumber\\
\mathcal{R}_2&=&~R_{\mathcal{M}}-R_{f},
\end{eqnarray*}
and
\begin{eqnarray*}
R_{\mathcal{M}}&=&\int^{t+\triangle t}_t(s-t)\big(\frac{d{\mathcal M}}{dt}
+v\cdot\nabla{\mathcal M}\big)(x_{\theta_1},v,t_{\theta_1})\cr
R_{f}&=&\int^{t+\triangle t}_t(s-t-\triangle t){\mathcal M}(f)(x_{\theta_2},v,t_{\theta_2})ds.
\end{eqnarray*}
\end{theorem}
\begin{proof}
Along the characteristic line, we have from \eqref{main.1}
\begin{equation*}
\frac{\partial f}{\partial t}(x+vt,v,t) = \frac{1}{\kappa}({\mathcal M}f-f)(x+vt,v,t).
\end{equation*}
We integrate in time from $t$ to $t+\triangle t$ to obtain
\begin{equation*}
f(x+(t+\triangle t)v,v,t+\triangle t)=f(x+tv,v,t)+\frac{1}{\kappa}\int^{t+\triangle t}_{t}({\mathcal M}f-f)(x+vs,v,s)ds,
\end{equation*}
or, equivalently,
\begin{eqnarray}
&&f(x,v,t+\triangle t)\label{Consist1}\\
&&\hspace{0.8cm}=f(x-\triangle tv,v,t)+\frac{1}{\kappa}
\int^{t+\triangle t}_{t}(~{\mathcal M}f-f~)(x-(t+\triangle t-s)v,v,s)ds\nonumber\cr
&&\hspace{0.8cm}\equiv f(x-\triangle tv,v,t)+ \frac{1}{\kappa}(~I_{\mathcal{M}}+I_{f}~)\nonumber.
\end{eqnarray}
Application of Taylor's theorem around $(x-\triangle t v,v,t)$ gives
\begin{eqnarray*}
{\mathcal M}(x-(t+\triangle t-s)v,v,s)&=&{\mathcal M}(x-\triangle tv,v,t)\\
&+&(s-t)v\cdot\nabla_x({\mathcal M}(x_{\theta_1},v,t_{\theta_1}))\\
&+&(s-t)\frac{d\mathcal {M}}{dt}(x_{\theta_1},v,t_{\theta_1}).
\end{eqnarray*}
where $x_{\theta_1}$ lies between $x-\triangle t v$ and $x-(t+\triangle t-s)$ and
$t_{\theta_1}$ $\in$ $[s,t]$.
Hence we have
\begin{equation}
I_{\mathcal{M}}=\triangle t{\mathcal M}(x-\triangle tv,v,t)+R_{\mathcal{M}}.\label{IM}
\end{equation}
where
\begin{eqnarray*}
R_{\mathcal{M}}&=&\int^{t+\triangle t}_t(s-t)v\cdot\nabla_x({\mathcal M}(x_{\theta_1},v,t_{\theta_1}))\\
               &+&\int^{t+\triangle t}_t(s-t)\frac{d\mathcal{ M}}{dt}(x_{\theta_1},v,t_{\theta_1}).\\
\end{eqnarray*}
On the other hand, we employ Taylor's theorem around $(x,v,t+\triangle t)$ to get
\begin{eqnarray*}
f(x-(t+\triangle t-s)v,v,s)&=&f(x,v,t+\triangle t)\\
&+&(s-t-\triangle t)v\cdot\nabla_x(f(x_{\theta_2},v,t_{\theta_2}))\\
&+&(s-t-\triangle t)\frac{df}{dt}(x_{\theta_2},v,t_{\theta_2}),
\end{eqnarray*}
where $x_{\theta_2}$ lies between $x$ and $x-(t+\triangle t-s)$ and
$t_{\theta_2}$ $\in$ $[s,t]$.
This gives
\begin{equation}
I_f=\triangle t f(x,v,t+\triangle t)+R_{f},\label{If}
\end{equation}
where
\begin{eqnarray*}
R_{f}&=&\int^{t+\triangle t}_t(s-t-\triangle t)v\cdot\nabla(f(x_{\theta_2},v,t_{\theta_2}))\\
               &+&\int^{t+\triangle t}_t(s-t-\triangle t)\frac{df}{dt}(x_{\theta_2},v,t_{\theta_2})\\
               &=&\int^{t+\triangle t}_t(s-t-\triangle t){\mathcal M}(f)(x_{\theta_2},v,t_{\theta_2}).
\end{eqnarray*}
for some appropriate $\theta_2\in$ $[t, t+\triangle t].$
Substituting (\ref{IM}) and (\ref{If}) into (\ref{Consist1}), we obtain
\begin{eqnarray*}
f(x,v,t+\triangle t)&=&f(x-\triangle tv,v,t)+\frac{\triangle t}{\kappa}{\mathcal M}(x-\triangle tv,v,s)-\frac{\triangle t}{\kappa} f(x,v,t+\triangle t)\\
&+&\frac{1}{\kappa}(~R_{\mathcal{M}}-R_{f}).
\end{eqnarray*}
We then collect relevant terms to have
\begin{align}
\begin{aligned}\label{Consist2}
f(x,v,t+\triangle t)&=\frac{\kappa}{\kappa+\triangle t}f(x-\triangle tv,v,t)+\frac{\triangle t}{\kappa+\triangle t}
{\mathcal M}(f)(x-\triangle tv,v,s)\cr
&+\frac{1}{\kappa+\triangle t}(~R_{\mathcal{M}}-R_{f})\cr
&=\frac{\kappa}{\kappa+\triangle t}\widetilde{f}(x,v,t)+\frac{\triangle t}{\kappa+\triangle t}
\widetilde{{\mathcal M}}(f)(x,v,s)\cr
&+\frac{1}{\kappa+\triangle t}(~R_{\mathcal{M}}-R_{f})\cr
&=\frac{\kappa}{\kappa+\triangle t}\widetilde{f}(x,v,t)+\frac{\triangle t}{\kappa+\triangle t}
{\mathcal M}(\widetilde{f})(x,v,s)\cr
&+\frac{\triangle t}{\kappa+\triangle t}\Big({\mathcal M}(\widetilde{f})(x,v,s)-\widetilde{\mathcal M}(f)(x,v,s)\Big)\cr
&+\frac{1}{\kappa+\triangle t}(~R_{\mathcal{M}}-R_{f}).
\end{aligned}
\end{align}
Finally, we put
\begin{eqnarray*}
\mathcal{R}_1&=&{\mathcal M}(\widetilde{f})(x,v,s)-\widetilde{\mathcal M}(f)(x,v,s),\nonumber\\
\mathcal{R}_2&=&~R_{\mathcal{M}}-R_{f}.
\end{eqnarray*}
to obtain the desired result.
\end{proof}
Before we estimate these remainder terms, we need to establish the following technical lemmas.
%
%
%
%
\begin{lemma}\label{ddtM} Let $f$ be a smooth solution of (\ref{main.1}) corresponding to an initial data $f_0$.
Then we have for $q\geq d+2$
\begin{eqnarray*}
\|\frac{d}{dt}(\rho, U, T)\|_{L^1}+\|\nabla_x(\rho, U, T)\|_{L^1}+\|v\cdot \nabla_x (\rho, U, T)\|_{L^1}\leq C_q N_q(f_0).
\end{eqnarray*}
\end{lemma}
\begin{proof}
We prove only the first estimate. Other estimates can be treated in a similar manner.
For $\Phi(x,v)=1,v,|v-U(x)|^2$,  we have
\begin{align}
\begin{aligned}
\Big|\int_{\bbr^d} f(x,v,t)\varphi(v) dv\Big|
&\leq N_q(f)\int_{\bbr^d}\frac{\varphi(v)}{(1+|v|)^q}dv\cr
&\leq C_qN_q(f_0)
\end{aligned}
\end{align}
and
\begin{align}
\begin{aligned}
\Big|\frac{d}{d t}\int_{\bbr^d} f(x,v,t)\varphi(v) dv\Big|
&=\Big|\int_{\bbr^d} \frac{\partial f}{\partial t}(x,v,t)\varphi(v) dv\Big|\cr
&=\Big|\int_{\bbr^d}  \Big(v\cdot \nabla f+\frac{1}{\kappa}({\mathcal M}-f)\Big)\varphi(v)dv\Big|\cr
&= \Big|\int_{\bbr^d}  v\cdot \nabla f\varphi(v) dv\Big|\cr
&\leq N_q(f)\int_{\bbr^d} |v|\frac{\varphi(v)}{(1+|v|)^q}dv\cr
&\leq C_qN_q(f_0).
\end{aligned}
\end{align}
Therefore, we have
\begin{align}
\begin{aligned}\label{dtrhoUT}
\big|\rho\big|, ~\big|\rho U\big|,~\big|\rho T\big| \leq C_q N_q(f_0),\cr
\Big|\frac{\partial \rho}{\partial t}\Big|, ~\Big|\frac{\partial (\rho U)}{\partial t}\Big|
,~\Big|\frac{\partial (\rho T)}{\partial t}\Big| \leq C_q N_q(f_0).
\end{aligned}
\end{align}
Then the result follows from (\ref{dtrhoUT}) and the lower bound estimates for local density and temperature given in Theorem \ref{Existence_Theorem}.
\end{proof}
%
%
%
%
\begin{lemma}\label{ddtM2} Let $f$ be a smooth solution corresponding to an initial data $f_0$.
Then we have for $q\geq d+2$
\begin{eqnarray*}
\|\frac{d{\mathcal M}}{dt}\|_{L^1_2}+\|\nabla_x{\mathcal M}\|_{L^1_2}+\|v\cdot \nabla_x {\mathcal M}\|_{L^1_2}\leq C_q N_q(f_0).
\end{eqnarray*}
\end{lemma}
\begin{proof}
We recall the chain rule:
\begin{eqnarray*}
\frac{d{\mathcal M}}{dt}
&=&\frac{\partial \rho}{\partial t}\frac{\partial {\mathcal M}}{\partial \rho}
+\frac{\partial U}{\partial t}\cdot\frac{\partial {\mathcal M}}{\partial U_i}
+\frac{\partial T}{\partial t}\frac{\partial {\mathcal M}}{\partial T},\cr
\nabla_x\mathcal M&=&\nabla_x\rho\frac{\partial M}{\partial \rho}
+\nabla_xU\cdot\frac{\partial M}{\partial U_i}
+\nabla_xT\frac{\partial M}{\partial T},\cr
\nabla_x\mathcal M&=&v\cdot\nabla_x\rho\frac{\partial M}{\partial \rho}
+v\cdot\nabla_xU\cdot\frac{\partial M}{\partial U_i}
+v\cdot\nabla_xT\frac{\partial M}{\partial T}
\end{eqnarray*}
and apply the estimates of the Lemma \ref{ddtM} to obtain
\begin{eqnarray*}
&&\|\frac{d{\mathcal M}}{dt}\|_{L^1_2}+\|\nabla_x{\mathcal M}\|_{L^1_2}+\|v\cdot \nabla_x {\mathcal M}\|_{L^1_2}\cr
&&\hspace{2cm}\leq C_qN_q(f_0)
\int_{\bbr^d}\Big(\Big|\frac{\partial M}{\partial \rho}\Big|+\Big|\frac{\partial M}{\partial U_i}\Big|
+\Big|\frac{\partial M}{\partial T}\Big|\Big)(1+|v|)^2  dv\cr
&&\hspace{2cm}\leq C_qN_q(f_0)\Big[(1+|U|^2+T)+\frac{\rho}{\sqrt{T}}(1+|U|^2+T)\Big]\cr
&&\hspace{2cm}\leq C_qN_q(f_0).
\end{eqnarray*}
This completes the proof.
\end{proof}
The following lemma provides the estimate of the remainder term, which will be crucial
for the convergence proof.
%
%
%
\begin{proposition}
$R_{\mathcal{M}}$, $R_f$ satisfies following estimates.
\begin{eqnarray}
\|\mathcal {R}_1\|_{L^1_2}&\leq& C_{T_f}\|f-\widetilde{f}\|_{L^1_2}+C_qN_q(f_0)\triangle t,\label{Remainder1}\\
\|\mathcal{R}_2\|_{L^1_2}&\leq& C_qN_q(f_0)\triangle t^2.\label{Remainder2}
\end{eqnarray}
\end{proposition}
\begin{proof}
We observe that
\begin{eqnarray*}
\|f-\bar f\|_{L^1_2}\leq \|\triangle t v\cdot \nabla_x f\|_{L^1_2}\leq C_qN_q(f)\triangle t.\label{This_PartIII_Chap1}
\end{eqnarray*}
This, with the estimate of Lemma \ref{ddtM2}, yields
\begin{eqnarray*}
\|\mathcal {R}_1\|_{L^1_2}&=&\|{\mathcal M}(\widetilde{f})-\widetilde{\mathcal M}(f)\|_{L^1_2}\\
&\leq&\|{\mathcal M}(\widetilde{f})-{\mathcal M}(f)\|_{L^1_2}
+\|{\mathcal M}(f)-\widetilde{\mathcal M}(f)\|_{L^1_2}\\
&\leq&C_{T_f}\|f-\widetilde{f}\|_{L^1_2}+\|\triangle t v\cdot\nabla_{x}{\mathcal{M}(f)}\|_{L^1_2}\\
&\leq&C_{T_f}\|f-\widetilde{f}\|_{L^1_2}+C_q N_q(f_0)\triangle t.
\end{eqnarray*}
On the other hand, we have by Lemma \ref{ddtM2}
\begin{eqnarray*}
\|\mathcal{R}_2\|_{L^1_2}&\leq&\int^{t+\triangle t}_t(s-t)\Big(~\|v\cdot \nabla\mathcal{M}(f)\|_{L^1_2}+\|\frac{d}{dt}\mathcal{M}(f)\|_{L^1_2}+\|{\mathcal M}(f)\|_{L^1_2}\Big)ds\\
&\leq&C_qN_q(f_0)\int^{t+\triangle t}_t(s-t)ds\\
&\leq&C_qN_q(f_0)(\triangle t)^2.
\end{eqnarray*}
\end{proof}

%
%
%
%
\section{{\bf Proof of the main result.}}
Before we delve into the proof of the main theorem, we need to establish some technical lemmas.
%
%
%
%
\begin{lemma}\label{Additional_Lemma1}Suppose $N_q(f_{0}),~N_q(f^0)<\infty$ with $q>d+2$, then we have
\begin{eqnarray*}
\|\widetilde{f}-\widetilde{f}^n_R\|_{L^1_2}\leq\|f-f^n_R\|_{L^1_2}+C_q\big(N_q(f_0)+N_q(f^0)\big)(\triangle x+\triangle v\triangle t).
\end{eqnarray*}\label{f-Ef}
\end{lemma}
\begin{proof}
We divide the estimate of $\|\widetilde{f}-\widetilde{f}^n_R\|_{L^1_2}$ into the following four integrals.
\begin{align}
\begin{aligned}\label{I+II}
&\|\widetilde{f}-\widetilde{f}^n_R\|_{L^1_2}\cr
&\hspace{0.6cm}=\int |f(x-v\triangle t,v,t)-\widetilde f^n_R(x,v)|(1+|v|)^2dxdv\cr
&\hspace{0.6cm}\leq\int|f(x-v\triangle t,v,t)-f(x-C_2(v)\triangle t,v,t)|(1+|v|)^2dxdv\cr
&\hspace{0.6cm}+\int |f(x-C_2(v)\triangle t,v,t)-f^n_R(x-C_2(v)\triangle t,v)|(1+|v|)^2dxdv\cr
&\hspace{0.6cm}+\int |f^n_R(x-C_2(v)\triangle t,v,t)-f^n_R(C_1(x)-C_2(v)\triangle t,v)|(1+|v|)^2dxdv\cr
&\hspace{0.6cm}+\int |f^n_R(C_1(x)-C_2(v)\triangle t,v,t)-\widetilde f^n_R(x,v)|(1+|v|)^2dxdv\cr
&\hspace{0.6cm}=I+II+III+IV.
\end{aligned}
\end{align}
By Taylor's theorem, $I$ can be treated as follows
\begin{eqnarray*}
I&=&\int|f(x-v\triangle t,v,t)-f(x-C_2(v)\triangle t,v,t)|(1+|v|)^2dxdv\\
&=&\int|v-C_2(v)|\triangle t|\partial_xf(x_{\theta},v,t)|(1+|v|)^2dxdv\\
&\leq&\triangle v\triangle t\int|\partial_xf(x_{\theta},v,t)|(1+|v|)^2dxdv\quad(\because ~|v-C_2(v)|\leq \triangle v)\\
&\leq&\triangle v\triangle tN_q(f)\int \frac{1}{(1+|v|)^q}(1+|v|)^2dxdv\\
&=&C\triangle v\triangle tN_q(f_0).
\end{eqnarray*}
On the other hand, we have by the change of variables with respect to $x$
\begin{eqnarray*}
II
&=&\int |f(x-C_2(v),v,t)-f^n_R(x-C_2(v)\triangle t,v)|(1+|v|)^2dxdv\\
&=&\int |f(x,v,t)-f^n_R(x,v)|(1+|v|)^2dxdv\\
&=&\|f-f^n_R\|_{L^1_2}.
\end{eqnarray*}
We now turn to the estimate of $III$. 
We define for simplicity $R_{ij}$ as
\begin{eqnarray*}
III&=&\int |f^n_R(x-C_2(v)\triangle t,v)-f^n_R(C_1(x)-C_2(v)\triangle t,v)|(1+|v|)^2dxdv\cr
&=&\int \sum_{i,j}\Big|\Big(\frac{x-v_j\triangle t-x_{s-1}}{\triangle x}f^n_{s,j,R}+
\frac{x_s-(x-v_j\triangle t)}{\triangle x}f^n_{s-1,j,R}\Big)\cr
&&-\Big(\frac{x_i-v_j\triangle t-x_{s}}{\triangle x}f^n_{s+1,j,R}+
\frac{x_{s+1}-(x_i-v_j\triangle t)}{\triangle x}f^n_{s,j,R}\Big)\Big|
(1+|v|)^2{\mathcal A}^R_{i,j}(x,v)dxdv\cr
&\leq&\sum_{i,j}\int \Big|\Big(\frac{x-v_j\triangle t-x_{s-1}}{\triangle x}f^n_{s,j,R}+
\frac{x_s-(x-v_j\triangle t)}{\triangle x}f^n_{s-1,j,R}\Big)\cr
&&-\Big(\frac{x_i-v_j\triangle t-x_{s}}{\triangle x}f^n_{s+1,j,R}+
\frac{x_{s+1}-(x_i-v_j\triangle t)}{\triangle x}f^n_{s,j,R}\Big)\Big|
(1+|v|)^2{\mathcal A}^R_{i,j}(x,v)dxdv\cr
&\equiv&\sum_{i,j}\int ~R_{ij}~
(1+|v|)^2{\mathcal A}^R_{i,j}(x,v)dxdv.
\end{eqnarray*}
We first observe that if $x\in [x_{i-1},x_i)$ and $x_i-v_j\triangle t $
$\in [x_s, x_{s+1})$, then we have
either $x-v_j\triangle t \in [x_{s-1}, x_{s})$ or $x-v_j\triangle t \in [x_s, x_{s+1})$. Hence we divide
the estimate into the following two cases.\newline
$(i)$ The case of $x-v_j\triangle t \in [x_{s-1}, x_{s})$:
For brevity, we put
\begin{eqnarray*}
a_{i,j}&=&\frac{x-v_j\triangle t-x_{s-1}}{\triangle x},\cr
b_{i,j}&=&\frac{x_i-v_j\triangle t-x_{s}}{\triangle x}
\end{eqnarray*}
to see
\begin{eqnarray*}
R_{i,j}&=&\big|~a_{i,j}f^n_{s,j,R}+(1-a_{i,j})f^n_{s-1,j,R}
-\big(~b_{i,j}f^n_{s+1,j,R}+(1-b_{i,j})f^n_{s,j,R}~\big)~\big|\cr
&\leq&b_{i,j}|f^n_{s+1,j,R}-f^n_{s,j,R}|+(1-a_{i,j})|f^n_{s,jR}-f^n_{s-1,j,R}|\cr
&\leq&\big\{b_{i,j}N^1_{q}(f^n_R)+(1-a_{i,j})N_q(f^n_R)\big\}\frac{\triangle x}{(1+|v_j|)^q}\cr
&\leq&2N_q(f^n_R)\frac{\triangle x}{(1+|v_j|)^q}.
\end{eqnarray*}
$(ii)$ The case of $x-v_j\triangle t \in [x_{s}, x_{s+1})$:
An almost identical argument gives
\begin{eqnarray*}
R_{i,j}\leq 2N_q(f^n_R)\frac{\triangle x}{(1+|v_j|)^q}.
\end{eqnarray*}
From $(i)$ and $(ii)$, we have
\begin{eqnarray*}
&&\int ~R_{ij}~(1+|v|)^2{\mathcal A}^R_{i,j}(x,v)dxdv\cr
&&\hspace{2cm}\leq 2N_q(f^n_R)\int^{v_{i+\frac{1}{2}}}_{v_{i-\frac{1}{2}}}\int^{x_{i+1}}_{x_i} \frac{\triangle x}{(1+|v_j|)^q}(1+|v|)^2dxdv\cr
&&\hspace{2cm}\leq 2N_q(f^n_R)\int^{v_{i+\frac{1}{2}}}_{v_{i-\frac{1}{2}}}\int^{x_{i+1}}_{x_i} \frac{\triangle x}{(1+|v_j|)^q}(1+|v_j+\frac{1}{2}\triangle v|)^2dxdv\cr
&&\hspace{2cm}\leq C_qN_q(f^n_R)\int^{v_{i+\frac{1}{2}}}_{v_{i-\frac{1}{2}}}\int^{x_{i+1}}_{x_i} \frac{\triangle x}{(1+|v_j|)^{q-2}}dxdv\cr
&&\hspace{2cm}\leq N_q(f^n_R)\frac{C_q(\triangle x)^2\triangle v}{(1+|v_j|)^{q-2}},
\end{eqnarray*}
which gives
\begin{eqnarray*}
III&\leq&C_qN_q(f^n_R)\sum_{i,j}\frac{(\triangle x)^2\triangle v}{(1+|v_j|)^{q-2}}\cr
&\leq&C_qN_q(f^n_R)\triangle x.
\end{eqnarray*}
Finally we estimate $IV$. Suppose $x_i-v_j \triangle t \in [x_s,x_{s+1})$, which implies
$x_{i+1}-v_j \triangle t \in [x_{s+1},x_{s+2})$. This gives for $(x,v)\in [x_{i-1},x_i)\times[v_{j-\frac{1}{2}}, v_{j+\frac{1}{2}})$
\begin{eqnarray*}
f^n_R(x_i-v_j\triangle t,v)=\frac{x_i-v_j\triangle t-x_s}{\triangle x}f^n_{s+1,j,R}+\frac{x_{s+1}-(x_i-v_j\triangle t)}{\triangle x}f^n_{s,j,R}
\end{eqnarray*}
and
\begin{eqnarray*}
\widetilde f^n_R(x,v)&=&\frac{x-x_i}{\triangle x}\widetilde f^n_{i+1,j,R}+\frac{x_{i+1}-x}{\triangle x}\widetilde f^n_{i,j,R}\cr
&=&\frac{x-x_i}{\triangle x}
\Big(\frac{x_{i+1}-v_j\triangle t-x_{s+1}}{\triangle x}f^n_{s+2,j,R}+\frac{x_{s+2}-(x_{i+1}-v_j\triangle t)}{\triangle x}f^n_{s+1,j,R}\Big)\cr
&+&\frac{x_{i+1}-x}{\triangle x}
\Big(\frac{x_{i}-v_j\triangle t-x_{s}}{\triangle x}f^n_{s+1,j,R}+\frac{x_{s+1}-(x_{i}-v_j\triangle t)}{\triangle x}f^n_{s,j,R}\Big).
\end{eqnarray*}
Hence we have for $(x,v)\in[x_{i-1},x_i)\times[v_{j-\frac{1}{2}}, v_{j+\frac{1}{2}})$
\begin{eqnarray*}
&&f^n_R(C_{1}(x)-C_2(v),v)-\widetilde f^n_R(x,v)\cr
&&\hspace{1cm}-\frac{x-x_i}{\triangle x}
\Big(\frac{x_{i+1}-v_j\triangle t-x_{s+1}}{\triangle x}f^n_{s+2,j,R}+\frac{x_{s+2,R}-(x_{i+1}-v_j\triangle t)}{\triangle x}f^n_{s+1,j,R}\Big)\cr
&&\hspace{1cm}+\frac{x-x_i}{\triangle x}
\Big(\frac{x_{i}-v_j\triangle t-x_{s}}{\triangle x}f^n_{s+1,j,R}+\frac{x_{s+1}-(x_{i}-v_j\triangle t)}{\triangle x}f^n_{s,j,R}\Big).
\end{eqnarray*}
Therefore, we have by a similar argument as in $III$
\begin{eqnarray*}
IV &\leq& C_qN_q(f^n_R)\sum_{i,j}\Big|\frac{x-x_i}{\triangle x}\Big|\frac{(\triangle x)^2\triangle v}{(1+|v_j|)^{q-2}}\cr
&\leq& C_qN_q(f^n_R)\triangle x,
\end{eqnarray*}
where we used $\Big|\frac{x-x_i}{\triangle x}\Big|\leq 1$.
Substituting these estimates into (\ref{I+II}), we obtain the desired result.
\end{proof}
%
%
%
%
\begin{lemma}\label{Additional_Lemma2}Suppose $N_q(f_{0R}),~N_q(f^0_R)<\infty$ with $q>d+2$, then we have
\begin{align}
\begin{aligned}
&(a)~\|{\mathcal M}(\widetilde{f_R})-\mathcal{M}(\widetilde{f}^n_R)\|_{L^1_2}\leq C_{T_f}\|f-f^n_R\|_{L^1_2}
+C_qN_q(f_{0R})(\triangle x+\triangle v\triangle t),\nonumber\cr
&(b)~\|{\mathcal M}(\widetilde{f}^n_R)-\mathcal{M}^n(\widetilde{f}^n_R)\|_{L^1_2}\leq C_qN_q(f_{0})\triangle v,\nonumber\cr
&(c)~\|{\mathcal M^n(\widetilde{f}^n_R)}-E\big({\mathcal M}^n(\widetilde{f}^n_R)\big)\|_{L^1_2}\leq C_qN_q(f_{0})(\triangle x+\triangle v).\nonumber
\end{aligned}\label{M-EM}
\end{align}
\end{lemma}
\begin{proof}
By Lemma \ref{ContinuityofMaxwellian} and Lemma \ref{Additional_Lemma1} $(a)$, we have
\begin{eqnarray*}
\|{\mathcal M}(\widetilde{f})-\mathcal{M}(\widetilde{f}^n_R)\|_{L^1_2}&\leq& C_{T_f}\|\widetilde{f}-\widetilde{f}^n_R\|_{L^1_2}\\
&\leq &C_{T_f}\|f-f^n_R\|_{L^1_2}+C_qN_q(f_0)(\triangle x+\triangle v\triangle t),
\end{eqnarray*}
which concludes (a).\\
We now prove (b). Let
\[
V_i=\{v_{i-\frac{1}{2}}\leq |v|<v_{i+\frac{1}{2}}\}.
\]
Then we have
\begin{eqnarray*}
&&\hspace{-1cm}{\mathcal M}(\widetilde{f}^n_R)-\mathcal{M}^n(\widetilde{f}^n_R)\\
&&=\frac{ {\widetilde\rho}^n_{R}}{\sqrt{(2\pi \widetilde {T}^n_{R}})^N}\exp(-\frac{|v-\widetilde {U}^n_R|^2}{2\widetilde {T}^n_R})
-\sum_{i}\frac{ {\widetilde\rho}^n_{R}}{\sqrt{(2\pi \widetilde {T}^n_{R}})^N}\exp(-\frac{|v-\widetilde {U}^n_R|^2}{2\widetilde {T}^n_R})
\mathcal{X}_{V_i}\\
&&\leq \sum_i|v-C_2(v)|\sup_{V_i}\Big|\frac{\partial \mathcal{M}(\widetilde{f}^n_R)}{\partial v}\Big|\mathcal{X}_{V_i}\\
&&\leq\triangle v\sum_i\sup_{V_i}\Big|\frac{\partial \mathcal{M}(\widetilde{f}^n_R)}{\partial v} \Big|\mathcal{X}_{V_i}.
\end{eqnarray*}
To estimate this, we recall Lemma \ref{ULBoundsMacroscopicDiscreteLemma} to see
\begin{eqnarray*}
\frac{\partial\mathcal{M}(\widetilde{f}^n_R)}{\partial v}&\leq& \frac{\rho^n_R}{\sqrt{(2\pi T^n_R)^d}}\frac{|v_{\theta}-U^n_R|}{\sqrt{T^n_R}}\exp\Big(-\frac{|v_{\theta}-U^n_R|^2}{2T^n_R}\Big)\\
&\leq&C_T\exp\Big(-\frac{|v_{\theta}-U^n_R|^2}{2T^n_R}\Big)\\
&\leq&C_T\exp(-C_T|v_i|^2),
\end{eqnarray*}
where we used
\begin{eqnarray*}
\exp\Big(-\frac{|v_{\theta}-U^n_R|^2}{2T^n_R}\Big)&=&
\exp\Big(-\frac{|v_i-(v_i-v_{\theta})-U^n_R|^2}{2T^n_R}\Big)\\
&\leq&\exp\Big(-\frac{|v_i|^2-2|(v_i-v_{\theta})+U^n_R|^2}{2T^n_R}\Big)\\
&\leq&\exp\Big(-\frac{|v_i|^2}{2T^n_R}\Big)\exp\Big(\frac{|\triangle v|^2+|U^n_R|^2}{2T^n_R}\Big)\\
&\leq&C_{T}\exp\Big(-C_T|v_i|^2\Big).
\end{eqnarray*}
Hence we have
\begin{eqnarray*}
&&\|{\mathcal M}(\widetilde{f}^n_R)-\mathcal{M}^n(\widetilde{f}^n_R)\|_{L^1_2}\leq C_qN_q(f_0)\triangle v.\\
\end{eqnarray*}

We now turn to (c), We consider
\begin{eqnarray*}
&&\|\mathcal{M}^n(\widetilde{f}^n_R)-E\big(\mathcal{M}^n(\widetilde{f}^n_R)\big)\|_{L^1_2}\\
&&\hspace{1cm}=\sum_{m,\ell}\Big\|\frac{x-x_m}{\triangle x}\Big(\mathcal{M}^n(\widetilde{f}^n_R)(x,v)-\mathcal{M}^n(\widetilde{f}^n_R)(x_{m+1},v_{\ell})\Big){\mathcal A}_{m,\ell}^R\Big\|_{L^1_2}\\
&&\hspace{1cm}+\sum_{m,\ell}\Big\|\frac{x_{m+1}-x}{\triangle x}\Big(\mathcal{M}^n(\widetilde{f}^n_R)(x,v)-\mathcal{M}^n(\widetilde{f}^n_R)(x_m, v_{\ell})\Big){\mathcal A}_{m,\ell}^R\Big\|_{L^1_2}\\
&&\hspace{1cm}=A+B.
\end{eqnarray*}
Since $\frac{x-x_i}{\triangle x}\leq 1$, we have
\begin{eqnarray*}
A&\leq&\sum_{m,\ell}\Big\|\mathcal{M}^n(\widetilde{f}^n_R)(x,v)-\mathcal{M}^n(\widetilde{f}^n_R)(x_{m+1}, v_{\ell}){\mathcal A}_{m,\ell}^R\Big\|_{L^1_2}\\
&=&\sum_{m,\ell}\int_{A_{m,\ell}^R}\Big|\mathcal{M}^n(\widetilde{f}^n_R)(x,v)-\mathcal{M}^n(\widetilde{f}^n_R)(x_{m+1}, v_{\ell})\Big|(1+|v|)^2dvdx\\
&\leq&C_qN_q(f_0)(\triangle x+\triangle v)\int\sum_{m,\ell}\frac{\mathcal{A}^R_{i,j}(x,v)}{(1+|v|)^{q-2}}dxdv\\
&\leq& C_qN_q(f_0)(\triangle x+\triangle v).
\end{eqnarray*}
Here we used for $(x,v)\in [x_m,x_{m+1})\times[v_{\ell-\frac{1}{2}},v_{\ell+\frac{1}{2}})$
\begin{eqnarray*}
&&\Big|\mathcal{M}^n(\widetilde{f}^n_R)(x,v)-\mathcal{M}^n(\widetilde{f}^n_R)(x_{m+1}, v_{\ell})\Big|\cr
&&\hspace{0.5cm}\leq
\Big|~(x-x_m)(\partial _x\rho+\partial_xU+\partial_xT)\cdot(\frac{\partial\mathcal{M}}{\partial\rho}+\frac{\partial\mathcal{M}}{\partial U}
+\frac{\partial\mathcal{M}}{\partial T})
+(v-v_{\ell})\frac{\partial\mathcal{M}}{\partial v}~\Big|\cr
&&\hspace{0.5cm}\leq\frac{N_q(f^n_R)}{(1+|v|)^{q-2}}(\triangle x+\triangle v),
\end{eqnarray*}
which can be obtained by an almost identical argument as the one given in the proof of Lemma \ref{ControlofMaxwellian}.
$B$ can be handled in a similar manner.
\end{proof}
We are now in a position to prove our main theorem.
\subsection{{\bf Proof of the main theorem}}
We subtract (\ref{refom_scheme}) from (\ref{Consist2}) and take $L^1_2$ norms to obtain
\begin{align}
\begin{aligned}\label{Reform-Consist}
\|f(\cdot,\cdot,(n+1)\triangle t)-f^{n+1}_R(\cdot,\cdot)\|_{L^1_2}&=\frac{\kappa}{\kappa+\triangle t}
\|\widetilde{f}(\cdot,\cdot,n\triangle t)-\widetilde{f}^{n}_R(\cdot,\cdot)\|_{L^1_2}\cr
&+\frac{\triangle t}{\kappa+\triangle t}
\|{\mathcal M}(\widetilde{f})-E\big({\mathcal M}(\widetilde{f}^n_R)\big)\|_{L^1_2}\cr
&+\frac{\triangle t}{\kappa+\triangle t}\|\mathcal{R}_1\|_{L^1_2}+\frac{1}{\kappa+\triangle t}\|\mathcal{R}_2\|_{L^1_2},
\end{aligned}
\end{align}
where
\begin{eqnarray*}
\mathcal{R}_1&=&{\mathcal M}(\widetilde{f})(x,v,s)-\widetilde{\mathcal M}(f)(x,v,s),\nonumber\\
\mathcal{R}_2&=&~R_{\mathcal{M}}-R_{f}.
\end{eqnarray*}
We then collect the estimates of Lemma \ref{Additional_Lemma1} and Lemma \ref{Additional_Lemma2}:
\begin{eqnarray}
\|\widetilde{f}-\widetilde{f}^n_R\|_{L^1_2}
&\leq&\|f-f^n_R\|_{L^1_2}+C_qN_q(f_0)\triangle v\triangle t+C_qN_q(f_0)\triangle x,\nonumber\cr
\vspace{0.4cm}
\|{\mathcal M}(\widetilde{f})-E\big({\mathcal M}(\widetilde{f^n_R})\big)\|_{L^1_2}
&\leq&\|{\mathcal M}(\widetilde{f})-\mathcal{M}(\widetilde{f}^n_R)\|_{L^1_2}
+\|{\mathcal M}(\widetilde{f}^n_R)-\mathcal{M}^n(\widetilde{f}^n_R)\|_{L^1_2}\nonumber\cr
&+&\|{\mathcal M^n(\widetilde{f}^n_R)}-E(\big({\mathcal M}^n(\widetilde{f}^n_R)\big))\|_{L^1_2}\nonumber\cr
&\leq&C\|f-f^n_R\|_{L^1_2}+C_qN_q(f_0)\triangle v\triangle t+C_qN_q(f_0)(\triangle v+\triangle x),\nonumber\cr
\vspace{0.4cm}
\|\mathcal{R}_1\|_{L^1_2}&\leq&C\|f-\widetilde{f}\|_{L^1_2}+C_qN_q(f_0)\triangle t,\nonumber\cr
|\mathcal{R}_2\|_{L^1_2}&\leq&CN_q(f_0)(\triangle t)^2.\nonumber
\end{eqnarray}
We substitute these estimates into (\ref{Reform-Consist}) to obtain
\begin{eqnarray*}
&&\|~f(\cdot,\cdot,(n+1)\triangle t) - f^{n+1}_R~\|_{L^1_2}\cr
&&\hspace{2cm}\leq \frac{\kappa }{\kappa+\triangle t}\Big(~\|f-f^n_R\|_{L^1_2}+C_{q,T_f}N_q(f_0)\triangle v\triangle t+C_{q,T}N_q(f_0)\triangle x\Big)\cr
&&\hspace{2cm}+\frac{C_{q,T_f}\triangle t}{\kappa+\triangle t}\Big(~\|f-f^n_R\|_{L^1_2}+N_q(f_0)\triangle v\triangle t+N_q(f_0)(\triangle v+\triangle x)\Big)\cr
&&\hspace{2cm}+\frac{C_{q,T_f}\triangle t}{\kappa+\triangle t}\Big(~\|f-f^n_R\|_{L^1_2}+N_q(f_0)\triangle t\Big)
+C_qN_q(f_0)\frac{1}{\kappa+\triangle t}(\triangle t)^2\cr
&&\hspace{2cm}=\Big(1+\frac{C_{T_f}\triangle t}{\kappa+\triangle t}\Big)\|~f - f^{n}_R~\|_{L^{\!\!1}_2}\cr
&&\hspace{2cm}+\Big(1+\frac{C_{T_f}\triangle t}{\kappa+\triangle t}\Big)\triangle v\triangle t\cr
&&\hspace{2cm}+\Big(1+\frac{C_{T_f}\triangle t}{\kappa+\triangle t}\Big)(\triangle x+\triangle v)\cr
&&\hspace{2cm}+\frac{C_{q,T_f}}{\kappa+\triangle t}(\triangle t)^2.
\end{eqnarray*}
We now put $\Gamma=\frac{C_T\triangle t}{\kappa+\triangle t}$ for simplicity of notation and
iterate the above inequality to obtain
\begin{align}
\begin{aligned}\label{a.f.estimate}\
\|f(\cdot,\cdot,N_t\triangle t)-f^{N_t}_R\|_{L^1_2}
&\leq(1+\Gamma)^{N_t}\|~f(\cdot,\cdot,0) - f^{0}_R~\|_{L^1_2}\cr
&+\sum_{i=1}^{N_t}(1+\Gamma)^i\triangle x\triangle t\cr
&+\sum_{i=1}^{N_t}(1+\Gamma)^i(\triangle x+\triangle v)\cr
&+\frac{C_{q,T_f}}{\kappa+\triangle t}\sum_{i=1}^{N_t-1}(1+\Gamma)^i(\triangle t)^2.\cr
\end{aligned}
\end{align}
Note that from the elementary inequality $(1+x)^n\leq e^{nx}$, we have
\begin{eqnarray*}
(1+\Gamma)^{N_t}&\leq&e^{N_t\Gamma }\\
&\leq& e^{\frac{C_{T_f}N_t\triangle t}{\kappa+\triangle t}}\\
&=& e^{\frac{C_{T_f}T_f}{\kappa+\triangle t}},
\end{eqnarray*}
and
\begin{eqnarray*}
\sum_{i=1}^{N_t}(1+\Gamma)^i&=&\frac{(1+\Gamma)^{N_t}-1-\Gamma}{(1+\Gamma)-1}\\
&\leq&C_{T_f}e^{\frac{C_{T_f}T_f}{\kappa+\triangle t}}\frac{\kappa+\triangle t}{\triangle t}.
\end{eqnarray*}
Similarly, we have
\begin{eqnarray*}
\sum_{i=0}^{N_t-1}(1+\Gamma)^i\leq C_{T_f}e^{\frac{C_{T_f}T_f}{\kappa+\triangle t}}\frac{\kappa+\triangle t}{\triangle t}.
\end{eqnarray*}
We then substitute these estimates into (\ref{a.f.estimate}) to obtain
\begin{align}
\begin{aligned}\label{a.f.estimate2}
\|f(\cdot,\cdot,N_t\triangle t)-f^{N_t}_R\|_{L^1_2}
&\leq e^{\frac{C_{T_f}{T_f}}{\kappa+\triangle t}}~\|~f(\cdot,\cdot,0) - f^{0}_R~\|_{L^1_2}\cr
&+ C_{T_f}e^{\frac{C_{T_f}{T_f}}{\kappa+\triangle t}}~(\kappa+\triangle t)~\triangle x\cr
&+ C_{T_f}e^{\frac{C_{T_f}{T_f}}{\kappa+\triangle t}}~(\kappa+\triangle t)\frac{(\triangle x+\triangle v)}{\triangle t}\cr
&+ C_{T_f}e^{\frac{C_{T_f}{T_f}}{\kappa+\triangle t}}~\triangle t.
\end{aligned}
\end{align}
To estimate $\|f_0-f^{0}_R\|_{L^1_2}$, we decompose the integral as
\begin{eqnarray*}
\|f_0-f^{0}_R\|_{L^1_2}&\leq&\|f_0-f^{0}_R\|_{L^1_2(|v|\leq R)}+\|f_0\|_{L^1_2(|v|\geq R)}.
\end{eqnarray*}
By direct estimates, we have
\begin{eqnarray*}
\|f_0-f^{0}_R\|_{L^1_2(|v|\leq R)}&=&\|f_0-E(f_0{\mathcal X}_{|v|\leq R})\|_{L^1_2}\cr
&\leq&C_q \overline{N}_q(f_0)(\triangle x+\triangle v).
\end{eqnarray*}
On the other hand, the second term can be estimated as follows
\begin{eqnarray*}
\|f_0\|_{L^1_2(|v|\geq R)}&\leq& N_q(f_0)\int_{|v|\geq R} \frac{1}{(1+|v|)^q}dv\\
&\leq& \frac{C_qN_q(f_0)}{(1+R)^{q+1}}.\\
\end{eqnarray*}
We substitute them into (\ref{a.f.estimate2}) and take
\[
C({T_f},q,f_0)=~\max\{C_{T_f}e^{\frac{C_{T_f}{T_f}}{\kappa+\triangle t}}, C_q \overline{N}_q(f_0)\},
\]
to obtain
\begin{eqnarray*}
\|f(\cdot,\cdot,T_f) - f^{N_t}_R\|_{L^1_2}\leq C(T_f,q,f_0)
\Big(\triangle x+\triangle v+\frac{\triangle v+\triangle x}{\triangle t}+\triangle t+\frac{1}{(1+R)^{q+1}}\Big).\\
\end{eqnarray*}
This completes the proof.\newline
\noindent{\bf Acknowledgment.} Part of this work was done while S.-B. Yun was visiting professor Giovanni Russo
at University of Catania. He would like to thank for their warm hospitality. S.-B. Yun's work was supported by the BK 21 program.

\end{document}